\documentclass[11pt,leqno]{amsart}
 \usepackage{ucs}
\usepackage{upgreek}
\usepackage{amsthm,amsfonts,amssymb,epsfig,graphics,cancel,esint,upgreek,
amsmath,eufrak,latexsym,enumerate,mathrsfs}
\usepackage[T1]{fontenc}
\usepackage{lmodern}
\usepackage{oldgerm}
\usepackage[latin1]{inputenc}\relax
\usepackage[active]{srcltx}
\usepackage[final]{pdfpages}
\usepackage{color}
\usepackage{setspace}
\definecolor{darkred}{rgb}{0.5,0,0}
\definecolor{darkgreen}{rgb}{0,0.5,0}
\definecolor{darkblue}{rgb}{0,0,0.5} 
\usepackage{graphicx,subfigure}
\usepackage[colorlinks = true, linkcolor=darkblue,
filecolor=darkgreen,
urlcolor=darkred,
citecolor=darkgreen,pdftex]{hyperref}
\usepackage{url}
\usepackage[small]{caption}
\usepackage{wrapfig}
\usepackage{enumitem}

\newcommand{\CC}{{\mathbb C}}

\newcommand{\RR}{{\mathbb R}}

\newcommand{\ZZ}{{\mathbb Z}}

\def\uu
\def\pp{\text {\rm \bf \c{p}}}

\def\cD{\mathcal{D}}
\def\cE{\mathcal{E}}
\def\cF{\mathcal{F}}

\def\cH{\mathcal{H}}

\def\cM{\mathcal{M}}

\def\cZ{\mathcal{Z}}

\def\mc{\mathfrak{c}}

\def\mi{\mathfrak{i}}

\def\mv{\mathfrak{v}}

\def\part{\partial}

\newcommand\id{\mathrm{Id} }

\newtheorem{theo}{Theorem}
\newtheorem{prop}[theo]{Proposition}

\newtheorem{lem}[theo]{Lemma}
\newtheorem{defi}[theo]{Definition}

\newtheorem{rem}[theo]{Remark}

\numberwithin{equation}{section}

\newcounter{numero}
\newcounter{numerob}
\newcounter{numerobb}
\usepackage{ulem}
\usepackage{cancel}

\usepackage{geometry}
\usepackage{amsmath}
\usepackage{amsfonts}
\usepackage{amsthm,amssymb}
\usepackage[latin1]{inputenc}  
\usepackage[T1]{fontenc}       
\usepackage{amssymb}

\title[Extended Magnetohydrodynamics]{ \bf The equations of \\
Extended Magnetohydrodynamics} 
\author[N. Besse]{Nicolas Besse}
\address[Nicolas Besse]{Observatoire de la C\^ote d'Azur,
               Bd de l'Observatoire CS 34229,
               06304 Nice Cedex 4, France}
\author[C. Cheverry]{Christophe Cheverry}
\address[Christophe Cheverry]{Institut Math\'ematique de Rennes, 
Campus de Beaulieu, 263 avenue du G\'en\'eral Leclerc CS 74205 
35042 Rennes Cedex\\FRANCE}

\begin{document}
\maketitle

\bigskip
\bigskip

\noindent \textbf{\small Abstract}. {\small Extended magnetohydrodynamics (XMHD) is a fluid plasma model \cite{GP04}
generalizing ideal MHD by taking into account the impact of Hall drift effects \cite{Light} and the influence of electron inertial 
effects \cite{Lus59}. XMHD has a Hamiltonian structure which has received over the past ten years a great deal of attention 
among physicists \cite{AKY15,Chari,DML16,LMM16,MLM17}, and which is embodied by a non canonical Poisson algebra 
on an infinite-dimensional phase space. XMHD can alternatively be formulated as a nonlinear evolution equation. Our aim 
here is to investigate the corresponding Cauchy problem. We consider both incompressible and compressible versions of 
XMHD with, in the latter case, some additional bulk (fluid) viscosity. In this context, we show that XMHD can be recast as a 
well-posed symmetric hyperbolic-parabolic system implying pseudo-differential operators of order zero acting as coefficients 
and source terms. Along these lines, we can solve locally in time the associated initial value problems, with moreover a minimal 
Sobolev regularity. We also explain the emergence and propagation of inertial waves \cite{ALM16,MLM17}.}

\bigskip
\medskip

\noindent \textbf{\small Keywords}. {\small Hyperbolic-parabolic symmetric systems of conservation laws; Initial value problem for nonlinear systems 
of PDEs; Partially elliptic systems; Compressible and incompressible fluid mechanics; 
Plasma physics; Hall, Inertial and Extended Magnetohydrodynamics; Pseudo-differential operators; Weyl quantization.}

\bigskip
\medskip

{\small \parskip=1pt
\setcounter{tocdepth}{1}
\tableofcontents
}

\medskip

$ \, $

\break


\section{Introduction} \label{Introduction} Extended magnetohydrodynamics (XMHD in abbreviated form) is a system of nonlinear evolution 
equations in the $ (3+1) $-dimensional spacetime $ \RR^3_\text{x} \times \RR_t $, which was first initiated by L\"ust \cite{Lus59}. It can 
be obtained \cite{AKY15} from a two-fluid model (electrons plus ions), under the assumptions of quasi-neutrality and smallness of the 
electron mass compared to the ion mass, by imposing an auxiliary ordering on the equations of motion; it can be recovered from kinetic 
theory \cite{JJM}; or, starting with a Lagrangian picture carried by some adequate two-fluid functional, it can be derived from action 
principles \cite{Chari,KMM17}. 

\smallskip

\noindent Extended MHD can be formulated as a (non-canonical) Hamiltonian system \cite{AKY15,DML16} which subsumes ideal MHD, 
Hall MHD, as well as inertial MHD models. It is equipped with a non canonical Poisson bracket \cite{DML16}, a conserved energy \cite{KimuM}, 
Casimir invariants and topological properties which are investigated in \cite{GTAM,LMM16} and references therein. 

\smallskip

\noindent Extended MHD is motivated by its great importance in astrophysics and geophysics. It has proven to be useful in several contexts, 
like solar wind \cite{ALM16} and neutron stars \cite{AY16}. It is also driven by nuclear fusion science since the Hall effect and the electron 
inertia are currently identified \cite{GTAM,Hossen} as potential reconnection mechanisms in collisionless plasmas. Now, reconnection is a 
dynamical process. Hence, the importance of developing the Eulerian approach in parallel to the aforementioned Lagrangian viewpoint.
This is precisely the position of this article, namely to explain how the XMHD evolution equations can be solved starting from initial data.

\smallskip

\noindent Normalizing variables in the standard Alfv\'en units, with the convention $ \nabla \equiv \nabla_\text{x} $, extended MHD is built with the continuity 
equation (on the total mass density $ \rho $ and the center-of-mass velocity $ \text{v} $),
\begin{equation}\label{eqcontinuityequation}   
\part_t \rho + \nabla \cdot (\rho \, \text{v}) = 0 \, , 
\end{equation}
and with the equation for the momentum density 
\begin{equation}\label{eqmomentumdensity}   
\rho \ \bigl(\part_t \text{v} + (\text{v} \cdot \nabla) \text{v} \bigr) + \nabla \text{p} - \jmath \times \text{B} + d_e^2 \ (\jmath \cdot \nabla) 
(\jmath / \rho) = 0 \, , 
\end{equation}
where $ \text{p} : \RR_+ \rightarrow \RR $ is a smooth function of $ \rho $ representing a pressure, $ \text{B} $ is the magnetic field and $\jmath $ is 
the current density. The two equations (\ref{eqcontinuityequation}) and (\ref{eqmomentumdensity}) must be completed with the Maxwell-Amp\`ere 
equation (where the displacement current $ \part_t \text{E} $ can be dropped under the assumption that our system is not relativistic) 
\begin{equation}\label{Maxwell-Ampere}  
 \jmath = \nabla \times \text{B} \, , 
 \end{equation}
with the Maxwell--Faraday equation
\begin{equation}\label{Maxwell-Faraday}   
\partial_t \text{B} + \nabla \times \text{E} = 0\,, 
 \end{equation}
and with a generalized Ohm's law \cite{KimuM}, which gives the electric field $ \text{E} $ in terms of the other unknowns $\rho $, $ \text{v} $, 
$ \text{B} $, $ \jmath $ and the electron pressure $ \text{p}_e $ according to (see \cite{AKY15,GTAM})
\begin{equation}\label{ohms}   
 \begin{array}{rl}
\text{E}+ \text{v} \times \text{B} = \! \! \! & \displaystyle  - \, \frac{d_i}{\rho} \ \nabla \text{p}_e + d_i \ \frac{\jmath}{\rho} \times \text{B} - 
d_i \ d_e^2 \ \Bigl( \frac{\jmath}{\rho} \cdot \nabla \Bigr) \frac{\jmath}{\rho} \smallskip \\
\ & \displaystyle  + \, d_e^2 \ \Bigl \lbrack \part_t \Bigl( \frac{\jmath}{\rho} \Bigr) + (\text{v}\cdot \nabla) \Bigl( \frac{\jmath}{\rho} \Bigr) + 
\Bigl( \frac{\jmath}{\rho} \cdot \nabla \Bigr) \text{v} \Bigr \rbrack \, .
\end{array} 
\end{equation}
The above two dimensionless parameters $ d_e $ and $ d_i $ are independent. They are non-negative ($ d_e \geq 0 $ and $ d_i \geq 0 $). 
They represent respectively the normalized \underline{e}lectron and \underline{i}on skin depths. In practice  (see Remark \ref{comparison}), 
they are often found to be adjusted in such a way that $ 0 \leq d_e \leq d_i \ll 1 $. Knowing that, the relation (\ref{ohms}) appears clearly as 
a perturbation of the ideal Ohm's law ($ \text{E}+ \text{v} \times \text{B} = 0 $).

\smallskip

\noindent In Subsection \ref{Mathematicalbackground}, we recall the state of knowledge concerning the mathematical results about ideal, Hall and 
extended MHD. In Subsections \ref{idealsituationIntroduction} and \ref{compressibleframeworkIntroduction}, we present our main outcomes concerning 
respectively the incompressible and compressible frameworks. This is also an occasion to outline the plan of the text and to emphasize some key 
ideas.


\subsection{Mathematical background} \label{Mathematicalbackground} Exploiting (\ref{Maxwell-Ampere}), we can express $ \jmath $ in terms 
of $ \text{B} $, and replace $ \jmath $ inside (\ref{eqmomentumdensity}) and (\ref{ohms}) accordingly. Then, we can substitute the electric field 
$ \text{E} $ thus obtained at the level of (\ref{Maxwell-Faraday}). When doing this, we can collect the quantities which undergo a time derivative,
namely
$$ \part_t \text{B} + d_e^2 \ \nabla \times \part_t \Bigl( \frac{\jmath}{\rho} \Bigr) = \part_t \Bigl \lbrack \text{B} + d_e^2 \ \nabla \times \Bigl( 
\frac{\nabla \times \text{B}}{\rho} \Bigr)  \Bigr \rbrack \, . $$
By this way, the expression
\begin{equation}\label{lienBB*inideb}  
\text{B}^* = \text{B} + d_e^2 \ \nabla \times \bigg( \frac{\nabla \times \text{B}}{\rho} \bigg) \, .
\end{equation}
acquires the status of a dynamical variable. The notation $ \text{B}^* $ is very common \cite{AKY15,GP04,LMM16}. From there, the unknowns 
are $ \rho $, $ \text{v} $ and $ \text{B}^* $, while the {\it constitutive relation} (\ref{lienBB*inideb}) is aimed to express $ \text{B} $ in terms of 
$ \text{B}^* $. After some calculations, or see directly the three equations (1)-(5)-(6) in \cite{LMM16}, we obtain the version of XMHD which is 
highlighted in \cite{AKY15,DML16,LMM16} and which is delivered in the form
\begin{equation}\label{compressibledeb}   
\left \lbrace \begin{array}{l}
\part_t \rho + (\text{v} \cdot \nabla )\rho + \rho \ \nabla \cdot \text{v} = 0 \, , \medskip \\
\displaystyle \part_t \text{v} + (\text{v} \cdot \nabla ) \text{v} + \frac{\nabla \text{p}}{\rho} + \text{B}^* \times \frac{\nabla \times \text{B}}{\rho} 
+ d_e^2 \ \nabla \bigg( \frac{ \vert \nabla \times \text{B} \vert^2}{2 \, \rho^2} \bigg) = \nu \ d_e^2 \ \nabla (\nabla \cdot \text{v})  \, , \medskip\\
\displaystyle \part_t \text{B}^* + \nabla \times \bigg( \text{B}^* \times \Bigl(\text{v} - d_i \, \frac{\nabla \times \text{B}}{\rho} \Bigr) \bigg) 
+ d_e^2 \ \nabla \times \bigg( (\nabla \times \text{v}) \times \frac{\nabla \times \text{B}}{\rho} \bigg) = 0 \, .
\end{array} \right. 
\end{equation}
In physics textbooks, these equations are often supplemented by
\begin{equation}\label{oublidivBB*}  
\nabla \cdot \text{B}^* = 0 \, , \qquad \nabla \cdot \text{B} = 0 \, .
\end{equation}
The equations (\ref{lienBB*inideb}) and (\ref{compressibledeb}) are derived (for $ \nu = 0 $)  in the contributions 
\cite{AKY15,ALM16,AY16,DML16,KimuM,LMM16} which mainly focus on the Hamiltonian formalism while the Eulerian approach is not really 
addressed. Individually, the equation (\ref{compressibledeb}) does not fall into usual mathematical categories and its well-posedness does not 
appear to have been clarified. In fact, the system (\ref{lienBB*inideb})-(\ref{compressibledeb}) looks like a quasilinear equation with various 
second order terms whose different roles need to be identified. The part $ \nu \ d_e^2 \ \nabla (\nabla \cdot \text{v}) $ where $ \nu > 0 $ represents 
a bulk (fluid) viscosity. It clearly provides some partial ellipticity on the component $ \text{v} $, namely a control on $ \nabla \cdot \text{v} $. But the 
other (nonlinear) second order terms (which are driven by $ d_e \geq 0 $ and $ d_i \geq 0 $) do not. Let us consider what can be said about the 
influence of $ d_e $ and $ d_i $.

\smallskip

\noindent First, assume that $ d_e = 0 $. Then, from (\ref{lienBB*inideb}), we deduce that $ \text{B}^* = \text{B} $, and two situations may be 
distinguished. For $ d_i = 0 $, we recover the equations of compressible MHD \cite{Maj84}, whose theory is today (almost) completed. 
For $d_i>0$, we incorporate the {\it Hall current term} coming from the third equation of (\ref{compressibledeb}) which (for $ \rho \equiv 1 $) 
reduces to the contribution $ d_i \, \nabla \times \bigl( (\nabla \times \text{B}) \times \text{B} \bigr) $. In particular, when $ \rho \equiv 1 $ and 
$ \nabla \cdot \text{v} = 0 $, we find the incompressible Hall-MHD system which has been introduced by Lighthill \cite{Light}. 

\smallskip

\noindent The situation $ d_e = 0 $ has been much studied by mathematicians in recent years, see for instance \cite{Chae,Dai,LiuTan,Ye}. This 
has been achieved in the presence of dissipative terms, namely a shear (fluid) viscosity ($ \mu \, \Delta \text{v} $ with $ \mu > 0 $) and/or a magnetic 
resistivity ($\eta\, \Delta \text{B} $ with $ \eta > 0 $). As soon as $ \eta > 0 $, the second order terms with $ d_i $ in factor can be absorbed, and the 
system becomes locally well-posed. But when $ \eta = 0 $, Hall-MHD equations are today known to be strongly ill-posed \cite{ChaeW}, even in Gevrey 
spaces \cite{JO}, and even if some kinematic viscosity $ \mu \, \Delta \text{v} $ with $ \mu > 0 $ is added. 

\smallskip

\noindent In this text, as prescribed by physicists \cite{AKY15,DML16,KimuM,LMM16,Lus59}, we work with $d_e>0$. This passage from $ d_e = 0 $ 
to $d_e>0$ is very significant since it allows to include {\it inertial effects} that are fundamental in plasma dynamics. We make progress in two directions:

\smallskip

\noindent $ \bullet $ Looking at the content of (\ref{compressibledeb}), this improvement (from $ d_e = 0 $ to $ d_e > 0 $) is already quite an achievement. 
Indeed, the situation $ d_e > 0 $ seems more complicated: the symmetric part of ideal MHD is broken (since $ \text{B} $ is substituted for $ \text{B}^* $); the 
Hall term (with its potential instabilities \cite{ChaeW,JO}) is still present; and there are extra nonlinear second order terms without evident sign conditions. 
Clearly, supplementary derivative losses may be expected, while the introduction of $ d_e $ does not furnish any dissipation. That is probably why the 
Cauchy problem associated with (\ref{compressibledeb}) has not yet (to our knowledge) been investigated.

\smallskip

\noindent $ \bullet $ In line with the preceding mathematical approaches, we use a touch of dissipation. We impose a bulk (fluid) viscosity $ \nu \, d_e^2 > 0 $. 
This condition is not demanding. In particular, it disappears when the flow is incompressible. The key highlight is, unlike \cite{Chae,Dai,LiuTan,Ye},  the absence 
of shear (fluid) viscosity ($ \mu \, \Delta \text{v} $ with $ \mu > 0 $) and magnetic resistivity ($\eta\, \Delta \text{B} $ with $ \eta > 0 $). This means that the Hall 
instabilities \cite{ChaeW,JO} can (locally in time) be compensated by inertial effects ($d_e>0$) without resorting to such additional dissipative terms. 

\smallskip

\noindent The question is why$ \, $? Our claim is that (\ref{compressibledeb}) becomes locally well-posed once $ d_e > 0 $ and (in the compressible case) 
once $ \nu > 0 $ for the following two principal reasons:
\begin{itemize}
\item [-] {\it About the influence of $ d_e > 0 $}. The analysis of derivative losses (when $ d_e = 0 $) does not include the constitutive relation 
(\ref{lienBB*inideb}) doing everything completely differently by modifying the role of $ \text{B} $ from $ \text{B} \equiv \text{B}^* $ (when 
$ d_e =0 $) to some another $ \text{B} \not \equiv \text{B}^* $ (when $ d_e > 0 $) with a gain of derivatives. Despite appearances, by a change 
of unknowns, the system (\ref{compressibledeb}) can be recast as a (foliation of) well-posed hyperbolic-parabolic systems (whose coefficients 
and source terms take the form of zero order pseudo-differential operators). In so doing, the inertial terms (those with $ d_e > 0 $ in factor) 
contribute to some (almost) symmetric structure, involving completely new features. In this interpretation, they do not provide second order 
dissipative perturbations. Instead, they contribute to the appearance of  inertial waves. 
\item [-] {\it About the influence of $ \nu > 0 $}. The introduction of a volume (fluid) viscosity (in place of a magnetic resistivity) is sufficient (and 
seems also necessary as in other contexts \cite{Masmou}) to absorb (for reasons like in \cite{KS}) the problematic contributions that remain 
in the compressible framework, when performing energy estimates.
\end{itemize}

\noindent As a consequence:
\begin{itemize}
\item [-] We will involve changes of variables that become singular when $ d_e \in \RR_+^* $ goes to zero. Throughout the text, it is therefore 
essential to work with $ d_e > 0 $. Keep in mind that there is no smooth passage from the case $ d_e > 0 $ to the case $ d_e = 0 $. For instance, 
in (\ref{lienBB*inideb}), $ \text{B} $ is expressed in terms of $ \text{B}^* $ through a (partially) elliptic operator which will prove to be (on some 
appropriate subspace) of order $ -2 $ when $ d_e > 0 $, and of order $ 0 $ when $ d_e = 0 $. 
\item [-] It is essential to assume that $ \nu > 0 $ when dealing with the compressible framework. 
\end{itemize}

\noindent Knowing that $ d_e > 0 $, we can prefer a rescaled version of (\ref{compressibledeb}) that makes us forget the role
of $ d_e $. To this end, we multiply $ \text{x} $, $ \text{v} $, $ \text{B}^* $ and $ \text{B} $ by $ d_e^{-1} $, while $ \text{p} $ is 
multiplied by $ d_e^{-2} $. In other words, we work with
$$ d := \frac{d_i}{d_e} \, , \qquad x := \frac{\text{x}}{d_e} \, , \qquad v := \frac{\text{v}}{d_e} \, , \qquad B^* := \frac{\text{B}^*}{d_e} \, , \qquad B := 
\frac{\text{B}}{d_e} \, , \qquad p := \frac{\text{p}}{d_e^2} \, . $$
With these conventions, we find
\begin{equation}\label{compressible}   
\left \lbrace \begin{array}{l}
\part_t \rho + (v\cdot \nabla) \rho + \rho \ \nabla \cdot v = 0 \, , \medskip \\
\displaystyle \part_t v + (v\cdot \nabla) v + \frac{\nabla p}{\rho} + B^* \times \frac{\nabla \times B}{\rho} + \nabla \bigg( 
\frac{ \vert \nabla \times B \vert^2}{2 \, \rho^2} \bigg) = \nu \ \nabla (\nabla \cdot v)  \, , \medskip\\
\displaystyle \part_t B^* + \nabla \times \bigg( B^* \times \Bigl(v - d \, \frac{\nabla \times B}{\rho} \Bigr) \bigg) + \nabla
\times \bigg( (\nabla \times v) \times \frac{\nabla \times B}{\rho} \bigg) = 0 \, , 
\end{array} \right. 
\end{equation}
together with 
\begin{equation}\label{lienBB*ini}  
B^* = B + \nabla \times \bigg( \frac{\nabla \times B}{\rho} \bigg) \, .
\end{equation}
Let $ \bar \rho \in \RR_+^* $ be a 
constant positive background density. At the initial time $ t = 0 $, we impose
\begin{equation}\label{inidata}  
\qquad (\rho, v,B^*) (0,\cdot) = (\bar \rho + \rho_0, v_0,B^*_0) \, . 
\end{equation}
We work away from vacuum, say with
\begin{equation}\label{awayvacuum}  
0 < \bar \rho /2 \leq \bar \rho + \rho_0 (x) \, . 
\end{equation}
Note that the parameter $ d_e $ is no more visible at the level of (\ref{compressible})-(\ref{lienBB*ini}). It is in fact hidden behind the 
definition of $ d $ and behind the preceding change of scales. It keeps of course some influence. Indeed, let $ (\rho, v,B^*) $ be a 
solution to (\ref{compressible})-(\ref{lienBB*ini}). We can adjust $ d_i $ in such a way that $ d_i = d \, d_e $ for a fixed $ d \geq 0 $,
and consider that $ d_e $ can vary. Coming back to the initial variables, we find that 
\begin{equation}\label{rescaledversion}  
(\rho, \text{v} , \text{B}^*) (t, {\rm x}) := (\rho,d_e \, v,d_e \, B^*) (t, {\rm x} / d_e) \, , \qquad d_e \in ]0,1] \, , 
\end{equation}
is a family of solutions to (\ref{compressibledeb}) which belongs (when $ d_e > 0 $ goes to $ 0 $) to a perturbative concentrating regime 
(the periodic regime will not be investigated here) near the constant solution $ (\bar \rho,0,0) $. Indeed, the velocity and magnetic 
components ($ \text{v} $ and $ \text{B}^* $) are of small amplitude $ d_e $ while the profiles are in $ \cH^s $ (and thus they are 
decreasing functions, typically compactly supported). Retain that $ d_e $ has a significant impact at the level of (\ref{compressibledeb}), 
when looking at $ (\rho, \text{v} , \text{B}^*)  $. However, $ d_e $ will not be apparent in our statements (which are uniform with respect to 
$ d_e $) since they are formulated in terms of (\ref{compressible})-(\ref{lienBB*ini}). 

\smallskip

\noindent As usual, we denote by $\cF $ the Fourier transform and by 
$$ \cH^s := \cF^{-1} \bigl( \langle \xi \rangle^s \, \cF \,L^2 \bigr) \, , \qquad s \in \RR \, , \qquad \langle \xi \rangle:=(1+|\xi|^2)^{1/2} \, , $$ 
the standard Sobolev--Bessel potential space. 


\subsection{The incompressible situation} \label{idealsituationIntroduction} This entails looking at the pressure $ p : \RR_+ \rightarrow \RR $ 
as a scalar function that plays the role of a Lagrange multiplier. This also requires to start with initial data as in (\ref{inidata}) with $ \rho_0 = 0 $ 
as well as 
\begin{equation}\label{inidataindivfree}  
\nabla \cdot v_0 = 0 \, , \qquad \nabla \cdot B_0^* = 0 \, .
\end{equation}
Equivalently (see Subsection \ref{ideal equations}), the incompressible situation implies that:
\begin{itemize}
\item [i)] The density $ \rho $ is a positive constant, say (without limiting the generality)
\begin{equation}\label{rhoconstant1}  
\rho = \bar \rho = 1 \, .
\end{equation}
\item [ii)] All the vector fields $ v $, $ B^* $ and $ B $ are solenoidal. They belong to $\cD^s$ with
$$ \cD^s := \bigl \lbrace \, D \in \cH^s (\RR^3;\RR^3) \, ; \, \nabla \cdot D = 0 \, \bigr \rbrace \, . $$
In other words, we work with
\begin{equation}\label{divfreeini}  
\nabla \cdot v = 0 \, , \qquad \nabla \cdot B^* = 0 \, , \qquad \nabla \cdot B = 0 \, .
\end{equation}
\item [iii)] The set (\ref{compressible}) of equations reduces to
\begin{equation}\label{syssimpli}  
\left \lbrace \begin{array}{l}
  \displaystyle \part_t v + (v\cdot \nabla ) v + \nabla p + B^* \times (\nabla \times B) = 0 \, , \medskip \\
\displaystyle \part_t B^* + \nabla \times \bigl( B^* \times (v - d \,\nabla \times B) \bigr) + \nabla \times \bigl( (\nabla \times v) 
\times (\nabla \times B ) \bigr) = 0 \, .
\end{array} \right. 
\end{equation}
\item [iv)] The constitutive relation is replaced by
\begin{equation}\label{lienBB*2}  
  B = (\id-\Delta)^{-1} B^* \, .
\end{equation}
\end{itemize}

\begin{theo}\label{wellposednessideal} [Local smooth wellposedness for incompressible XMHD] Fix the initial data such that
\begin{equation}\label{reginidatain}  
(v,B^*) (0,\cdot) = (v_0,B^*_0) \in \cD^{s} (\RR^3;\RR^3) \times \cD^{s-1} (\RR^3;\RR^3) \, , \qquad s > 5/2  \, .
\end{equation}
Then, we can find some time $ T > 0 $ depending only on the $ \cH^s \times \cH^{s-1} $-norm of $ (v_0,B^*_0) $ such that the Cauchy 
problem built with (\ref{divfreeini})-(\ref{syssimpli})-(\ref{lienBB*2}) together with the initial condition (\ref{reginidatain}) has a 
unique local solution on $ [0,T] $, which is smooth in the following sense
\begin{equation}\label{reginidatainpropa}   
(v,B^*,B) \in  C \bigl( [0,T] ; \cD^s (\RR^3;\RR^3) \times \cD^{s-1} (\RR^3;\RR^3) \times \cD^{s+1} (\RR^3;\RR^3)\bigr) \, . 
\end{equation}
\end{theo}

\noindent In (\ref{reginidatainpropa}), the level $ \cH^s $ of regularity for $ v $ does not match with the one $ \cH^{s-1} $ obtained for the dynamical 
variable $ B^* $. This is because the presentation (\ref{syssimpli}), though inherited from physics, is not suitable from the perspective 
of initial value problems. This explains probably why things have not yet worked in this way. To remedy this, we transform in 
Section \ref{vorcasincompressible} the equations of (\ref{syssimpli}). More precisely, we incorporate a new equation on the 
vorticity $ w := \nabla \times v $ in order to obtain a system on $ (w,B^*) $ called the {\it vorticity formulation}. Then, we derive energy 
estimates up to the proof of Theorem \ref{wellposednessideal} (for $ s > 7/2 $).


\subsection{The compressible framework} \label{compressibleframeworkIntroduction} We present below our result concerning
(\ref{compressible})-(\ref{lienBB*ini}). As a prerequisite, we assume a barotropic equation of state. In other words, the pressure 
$ p $ is only a function of the density $ \rho $. It is prescribed by a smooth given function $ p : \RR_+ \rightarrow \RR $ whose 
derivative $ p' $ is positive. 

\begin{theo}[Local smooth wellposedness for compressible XMHD] \label{theoprin} Assume that $ \nu > 0 $, and fix any $ s > 5/2 $.
Select some initial data as in (\ref{inidata})-(\ref{awayvacuum}), with moreover 
\begin{equation}\label{reginidata}  
\quad ( \rho_0 ,  v_0 , B_0^* ) \in \cH^s (\RR^3;\RR) \times \cH^s (\RR^3;\RR^3) \times \cH^{s-1} (\RR^3;\RR^3) \, , \qquad \nabla \cdot B_0^* 
\in \cH^{s} (\RR^3;\RR^3) \, .
\end{equation}
Then, we can find some time $ T > 0 $ which is proportional to the parameter $ \nu $ and inversely proportional to the $ \cH^s \times \cH^s \times \cH^{s-1} 
\times \cH^s $-norm of $ (\rho_0,v_0,B^*_0, \nabla \cdot B^*_0) $ such that the Cauchy problem built with (\ref{compressible})-(\ref{lienBB*ini}) 
together with (\ref{inidata})-(\ref{awayvacuum})-(\ref{reginidata}) has a unique local solution on $ [0,T] $, which is smooth in the following sense 
\begin{equation}\label{sensesmooth}  
 \quad (\rho,v,B^*, B) \in C \bigl([0,T];\cH^s (\RR^3;\RR) \times \cH^s (\RR^3;\RR^3) \times \cH^{s-1} (\RR^3;\RR^3) \times \cH^{s+1} 
(\RR^3;\RR^3) \bigr) \, . 
\end{equation}
\end{theo}

\noindent Let us suppose, as it is often the case in practice, that $ 0 < d_e \leq d_i \ll 1 $. Then, our analysis indicates that the system 
(\ref{compressibledeb}) involves a mix of three interconnected regimes:
\begin{itemize}
\item [-] For low frequencies $ \vert \xi \vert \lesssim 1 $, the solutions behave (approximately) as provided for by compressible magnetohydrodynamics 
\cite{Ala08,Maj84,Scho07}. 
\item [-] For intermediate frequencies $ d_i^{-1} \lesssim \vert \xi \vert \ll d_e^{-1} $, the Hall effects come into play \cite{Chae,Dai,LiuTan,Light,Ye}, 
and various amplification mechanisms become to be implemented. This includes a step towards the singularity formations detected by 
mathematicians \cite{ChaeW,JO} and the tearing modes studied by physicists \cite{FP,GTAM} in the perspective of collisionless magnetic 
reconnection. However, in the weakly nonlinear regime (\ref{rescaledversion}) and as long as the time remains finite, these instabilities do 
not induce the explosion (of norms)  and they do not jeopardize the construction of solutions.
\item [-] For large frequencies $ d_e^{-1} \lesssim \vert \xi \vert $, inertial aspects take the place and new speeds (modes) of propagation 
appear. This means the emergence of {\it inertial waves} (see Paragraph \ref{The inertial waves}), whose impacts have been already observed 
by physicists \cite{ALM16,MLM17} but which do not seem to have been mathematically well identified before. 
\end{itemize}

\smallskip

\noindent It should be borne in mind that Theorem \ref{wellposednessideal} is more accessible than Theorem \ref{theoprin}. 
To some extent, it can be viewed as a simplified version of it. This is why the analysis begins in Section \ref{vorcasincompressible} 
with completing the incompressible situation. This makes the basic ideas more accessible. This also furnishes clear guidelines in the perspective 
of the compressible framework which is investigated in Section \ref{vorcompressibleframework}.

\break

\noindent Section \ref{vorcompressibleframework} follows the same steps as in Section \ref{vorcasincompressible} but it faces new 
challenges:
\begin{itemize}
\item [-]  On the one hand, in comparison with (\ref{lienBB*2}), due to the variations of $ \rho $, it is more difficult to exhibit the 
properties of (partial) ellipticity which are hidden behind the constitutive relation (\ref{lienBB*ini}), see Subsection \ref{transformationcompressible}. 
\item [-] On the other hand, the incompressible transformation must be adapted to the compressible framework, see Subsection 
\ref{transbiscompressible}. We still add the vorticity $ w = \nabla \times v $. Besides, we implement the divergence $ \nabla \cdot v $ and one 
order derivatives of $ \rho $. The system thus obtained is called the {\it compressible vorticity formulation}. 
\item [-] In subsection \ref{lwellposedcomp}, we remark that the divergence of $ B^* $ is a preserved quantity. Taking advantage of this 
information, we show that there is no loss of hyperbolicity and that energy estimates become available.  This is the entry point to the proof 
of Theorem \ref{theoprin}  (at least for $ s > 7/2 $). 
\end{itemize}

\noindent Another salient point should be reported. When dealing in space dimension $ d=3 $ with Sobolev solutions to quasilinear 
systems, the restriction $ s -1> 1+ (d/2) = 5/2 $ (or equivalently $ s > 7/2 $) on the component $ B^* $ would be expected \cite{Maj84,Scho07}. 
In Theorems \ref{wellposednessideal} and \ref{theoprin}, observe the presence of the relaxed condition $ B^* \in \cH^{(3/2)+} $ instead of 
the usual constraint $ B^* \in \cH^{(5/2)+} $. There is a gain of one degree of regularity which is justified in Section \ref{potentialformulations}. 
To this end, instead of looking at derivatives of $ (\rho,v) $, we integrate the magnetic field $ B^* $. As a matter of fact, we consider the 
magnetic potential $ A^* $ which is such that $ \nabla \times A^* = B^* $ and $ \nabla \cdot A^* = 0 $. This leads to the {\it potential formulation}.

\smallskip

\noindent The potential formulation furnishes a self-contained system on $ (\rho,v,A^*) $, which can be studied independently 
and which furnishes different types of supplementary information. This corresponds to the most completed approach but also 
in some aspects to the most challenging. This is why it is explained lastly. In fact, the potential formulation falls (modulo adaptations) 
under the scope of Kawashima-Shizuta theory \cite{KS}. This leads to the optimal regularity results (with $ s > 5/2 $) stated in Theorems 
\ref{wellposednessideal} and \ref{theoprin}.

\smallskip

\noindent Section \ref{Extended dispersion relations} is to exhibit the various types of inertial waves that can arise, and to study their 
properties. To this end, we first select special solutions (constant, in the form of Beltrami fields, corresponding to null point configurations, 
two dimensional, or even moving). Then, we look at the associated linearized equations and we focus on the regime of high frequencies 
(with $ d_e^{-1} \lesssim \vert \xi \vert $). By this way, we can highlight the presence of {\it inertial dispersion relations} which are of 
particular interest. 

\smallskip

\noindent There is a short Appendix, in Section \ref{Appendix}. It is about the div-curl system which appears repeatedly
throughout the text.

\smallskip

\noindent Given a state variable $ U $, we often employ the notation $ U^\star_\diamond $. The superscript $ \star \in \{ \mi, \mc \} $ 
is to indicate that $ U $ is related respectively to the \underline{$ \mi $}ncompressible and \underline{$ \mc $}ompressible situations. 
The subscript $ \diamond \in \{ \mv,\mathfrak{p} \} $ (where $ \mv $ and $ \mathfrak{p} $ must not be confused with the velocity $ v $
and the pressure $ p $) refers to the \underline{$ \mv $}orticity 
and \underline{$ {\mathfrak p} $}otential formulations. We reserve the {\it rsfs} font $ \mathscr P $ for operators, with a symbol denoted by the 
standard font $ P $, so that $ \mathscr P = P(D_x) $. We often put the subscript $ * \in \ZZ $ to specify that $ \mathscr P^\star_* = 
P^\star_*(D_x)  $ is of (maximal) order $ * $, while the superscript $ \star \in \{ \mi, \mc \} $ may still be incorporated for the same 
reasons as before.
 

\section{The incompressible situation} \label{vorcasincompressible}
In Subsection \ref{ideal equations}, we introduce the incompressible equations and some of its principal features. In Subsection \ref{idealconstitutive}, 
we exhibit properties of ellipticity lying behind (\ref{lienBB*2}). In Subsection \ref{transbis}, we perform a dependent change of 
unknowns which transforms (\ref{syssimpli}). In Subsection \ref{lwellposedin}, we derive energy estimates in order to show Theorem 
\ref{wellposednessideal}.


\subsection{The incompressible equations} \label{ideal equations}
The incompressible situation is strongly linked to the system (\ref{compressible})-(\ref{lienBB*ini}) of origin. To see how, starting from 
(\ref{compressible})-(\ref{lienBB*ini}), we have to deduce (\ref{rhoconstant1}), (\ref{divfreeini}), (\ref{syssimpli}) and (\ref{lienBB*2}). 
To this end, we consider below successively the indents i), $ \cdots $, iv)  of Subsection \ref{idealsituationIntroduction}.
\begin{itemize}
\item [i)] Since $\nabla \cdot v = 0$, the first equation of (\ref{compressible}) implies that the density $\rho$ 
is just advected along the characteristic curves generated by the vector field $v$. Hence it remains constant, 
say $ \rho = \bar \rho = 1$, if initially $\rho_0 = 0 $. 
\item [ii)] As already explained, the term $ \nabla p $ plays the role of a Lagrange multiplier which ensures the 
propagation of the constraint $ \nabla \cdot v = 0 $. On the other hand, it is clear that the divergence-free condition 
imposed inside (\ref{inidataindivfree}) on $ B^* $ at time $ t = 0 $ is propagated via the divergence of the third 
equation of \eqref{compressible}, and that it is transmitted to $ B $ through (\ref{lienBB*ini}). 
\item [iii)] We can always incorporate the part $ \vert \nabla \times B \vert^2/ {2 \, \rho^2} $ to the function $ p $. Then, 
knowing that $ \rho = 1 $ and $ \nabla \cdot v = 0$, the system (\ref{compressible}) is exactly the same as (\ref{syssimpli}).
\item [iv)] The link between $ B $ and $ B^* $ is here simplified into $ B^* = B + \nabla \times ( \nabla \times B) $.
Now, since $ \nabla \cdot B = 0 $, we have $ \nabla \times ( \nabla \times B) = - \Delta B $. After inversion, this yields (\ref{lienBB*2}).
\end{itemize}

\noindent Before proceeding, we remark that there is a conserved quantity which may be expressed in terms of $ (v,B) $.

\begin{lem} \label{energypreserin} [A conserved quantity]
incompressible XMHD preserves the energy 
  \begin{equation}    \label{conservedprop}  
 \cE^\mi := \frac 1 2\int_{\RR^3} \big(\vert v \vert^2 + \vert B \vert^2 + \vert \nabla \times B \vert^2 \big) \, dx \, .
   \end{equation}
  \end{lem} 

\begin{proof} Take the $L^2$-scalar product of the first equation of \eqref{syssimpli} with $v$. Using integration by parts and the 
condition $\nabla \cdot v=0$, all terms vanish except $ B^* \times (\nabla \times B)$ giving
\begin{equation}
  \label{eqn_nrj_1}
  \frac{d}{dt} \bigg( \frac 12\int_{\RR^3} |v|^2 \, dx \bigg) + \int_{\RR^3} v\cdot \bigl( B^* \times (\nabla \times B) \bigr) \, dx = 0 \,.  
\end{equation}
Take the $L^2$-scalar product of the second equation of \eqref{syssimpli} with $B$ (but not $ B^* $). Perform integration by parts 
(or exploit that the curl operator is self-adjoint), to see that the two triple products vanish. There remains
  \begin{equation}
  \label{eqn_nrj_2}
  \frac{d}{dt} \bigg( \frac 12\int_{\RR^3} \big(|B|^2 + |\nabla \times B|^2\big) \, dx \bigg)  + \int_{\RR^3} (\nabla \times B) \cdot ( B^* \times v)\, dx=0\,.  
\end{equation}
Summing \eqref{eqn_nrj_1} and \eqref{eqn_nrj_2}, we obtain that $d \cE^\mi/dt=0$ as expected. 
\end{proof}

  \begin{rem}\label{Leray} [Similarities with Leray--$\alpha$ models]
    Incompressible XMHD equations may bear some resemblance to Lagrangian averaged 
    (or Leray--$\alpha$) Euler equations \cite{HL06,MSH01,MSH03}, where a parameter $\alpha$ is introduced and represents 
    the spatial scale below which the dymanics are averaged. But if the parameter $d_e$ can be seen (to some extent) as playing 
    the part of $\alpha$ in Lagrangian averaged $\alpha$--models, its introduction is driven by other considerations related to 
    two-fluid models \cite{GP04,JJM} and its handling is completely different.
  \end{rem}

\noindent We also come back to the introduction of $ d $, and its significance.

\begin{rem} \label{comparison} [Comparison of electron and ion skin depths] Let $ \omega_{pi} $ and $ \omega_{pe} $ be the ion and
electron plasma frequencies. The ratio between $ d_i $ and $ d_e $ can be expressed in terms of 
\href{https://en.wikipedia.org/wiki/Plasma_parameters}{plasma parameters} according to
$$ d = \frac{d_i}{d_e} = \frac{\omega_{pe}}{\omega_{pi}} \simeq 42,72 \ \frac{\sqrt n_e}{\sqrt n_i} \ \frac{\sqrt \mu}{Z} \, ,  $$
where $ n_i $ and $ n_e $ are the \href{https://en.wikipedia.org/wiki/Number_density}{number densities} of ions and electrons, $ \mu =
m_i/m_p $ is the ion mass (expressed in units of the proton mass), and $ Z = q_i/e $ equals to the 
\href{https://en.wikipedia.org/wiki/Plasma_parameters}{atomic number}. In most plasmas, the parameters $ d_e $ and $ d_i $ are 
such that $ d_e \ll d_i $ or even $ d_e \leq 0,1 \times d_i $. However, these two parameters are completely independent, and there are situations 
(related to magnetic reconnection, see the paragraph  "{\it Inertial MHD}" in \cite{LMM16}-p.2402) where $ d_i $ and $ d_e $ could be 
comparable, with $ d_i \sim d_e $. 
\end{rem} 


\subsection{The incompressible constitutive relation} \label{idealconstitutive}
The operator 
$$ \nabla \times \nabla \times :  \cH^2(\RR^3;\RR^3) \rightarrow L^2(\RR^3;\RR^3) $$ 
is not elliptic of order $ 2 $, since it has a nonzero kernel. To avoid this difficulty, it suffices to restrict its action on a suitable subspace.

\begin{lem} \label{ellde2} [Underlying ellipticity when passing from $ B^* $ to $ B $ through the relation (\ref{lienBB*2})] The differential operator 
\begin{equation}\label{differentialoperator}   
\mathscr L^\mi_2  := \id + \nabla \times \nabla \times : \cD^s \rightarrow \cD^{s-2} \, , \qquad s \in \RR \, , 
\end{equation}
is bijective and elliptic of order $ 2 $. Its inverse $ (\mathscr L^\mi_2)^{-1} : \cD^{s-2} \rightarrow \cD^{s} $ takes the form of a Fourier multiplier 
which is elliptic of order $ -2 $.
\end{lem} 

\begin{proof} Let $ \bigl( e_1(\xi), e_2(\xi), e_3(\xi) \bigr) $ be a smooth orthonormal frame on $ \RR^3 \setminus \{0\} $ which is adjusted such 
that $ e_1(\xi) = \xi / \vert \xi \vert $. Let $ O_0 (\xi) $ be the orthogonal matrix whose column vectors are $ e_1(\xi) $, $ e_2(\xi) $ and $ e_3(\xi) $. 
In other words
\begin{equation}\label{defdePxi}  
O_0 (\xi) := \bigl( e_1(\xi), e_2(\xi), e_3(\xi) \bigr) \, , \qquad e_1(\xi) = \xi / \vert \xi \vert \, , \qquad e_i(\xi) \cdot e_j(\xi) = \delta_{ij} \, .
\end{equation}
Since $ \xi \times e_1(\xi) = 0 $ whereas $ \xi \times \xi \times e_j (\xi) = - \vert \xi \vert^2 \, e_j (\xi) $ for $ j \in \{2,3 \} $, we have
\begin{equation}\label{fourierac}  
\qquad \cF \, (\id + \nabla \times \nabla \times ) \, \cF^{-1} = O_0 (\xi) \, D_2^\mi  (\xi) \, O_0(\xi)^{-1} , \qquad D_2^\mi  (\xi) := \text {\footnotesize 
$ \displaystyle 
\left( \begin{array}{ccc}
1 & 0 & 0 \\
0 & \langle \xi \rangle^2 & 0 \\
0 & 0 & \langle \xi \rangle^2 
\end{array} \right) $} . 
\end{equation}
In other words, the action of $ \id + \nabla \times \nabla \times $ on the whole space $ \cH^2 (\RR^3;\RR^3) $ is unitary equivalent through a
conjugation by $ \mathscr O_0 := O_0 (D_x) $ to the diagonal operator $ \mathscr D^\mi_2 := D^\mi_2 (D_x) $. Introduce the $ L^2 $-projectors 
$ \mathscr P $ and $ \mathscr Q $, where $ \mathscr Q := \id - \mathscr P $ and $ \mathscr P = P(D_x) $ is given by the Leray projector whose 
matrix valued symbol is given by
\begin{equation}\label{l2projector}   
P (\xi) \, v :=  \bigl( e_2(\xi) \cdot v \bigr) \, e_2 (\xi) + \bigl(e_3(\xi) \cdot v \bigr) \, e_3 (\xi) \, . 
\end{equation}
Recall that the operator $ \mathscr L^\mi_2 $ is defined by the restriction of its action to $ \cD^s $, while the set $ \cD^s $ may be characterized by
\begin{equation}\label{charactds}   
\cD^s := \mathscr P \, \cH^s (\RR^3;\RR^3) \, .
\end{equation}
This implies that $ \mathscr L^\mi_2 $ does not see the eigenvalue $ 1 $ of $ D^\mi_2 (\xi) $. It just acts on the Fourier side according 
to the multiplier
\begin{equation}\label{reducdeLe}  
\quad \cF \, \mathscr L^\mi_2 \, \cF^{-1} = \bigl( \cF \, (\id + \nabla \times \nabla \times) \, \cF^{-1} \bigr)_{\mid \cF \, \cD^s} \equiv \langle \xi \rangle^2 \ \id \, ,
\qquad \mathscr L^\mi_2 \equiv \id - \Delta \, , 
\end{equation}
which is bijective and elliptic of order $ 2 $. From (\ref{reducdeLe}), we infer that 
$$ \cF \, (\mathscr L^\mi_2 )^{-1} \, \cF^{-1} = \langle \xi \rangle^{-2} \, \id \, , \qquad (\mathscr L^\mi_2)^{-1} \equiv (\id - \Delta)^{-1} \, . $$ 
This clearly confirms that resorting to $ (\mathscr L^\mi_2)^{-1} $ allows to gain two derivatives.
\end{proof}

\noindent With the above convention, we can deduce from (\ref{lienBB*2}) the incompressible constitutive relation
\begin{equation}\label{ell-1}   
\nabla \times B = \mathscr K^\mi_{-1} \, B^* \, , \qquad \mathscr K^\mi_{-1} := (\mathscr L^\mi_2 )^{-1} \, \nabla \times \equiv \mathscr 
K^\mi_{-1} \mathscr P \, .  
\end{equation}
This relation and Lemma \ref{ellde2} are essential because they allow to interpret all terms implying $ \nabla \times B $ inside (\ref{syssimpli})
as acting on $ B^* $ like operators of order $ -1 $ (instead of $ 1 $ when $ d_e = 0 $). This means that the expressions $ B $ and $ B^* $ do 
not play similar roles. At the same time, this invites to reconsider the hierarchy of terms when looking at (\ref{syssimpli}). With this in mind, 
in the next subsection, we apply the curl operator on the first equation of (\ref{syssimpli}). 


\subsection{Transformation of the incompressible equations} \label{transbis}
The purpose of this subsection is twofold. First, in Paragraph \ref{ideal vorticity formulation}, we exploit (\ref{divfreeini}) and 
(\ref{lienBB*2}) in order to recast (\ref{syssimpli}). Secondly,  in Paragraph \ref{The inertial waves}, we give a concrete meaning 
to the notion of inertial waves.


\subsubsection{The incompressible vorticity formulation} \label{ideal vorticity formulation}
The point is to implement the vorticity $ w := \nabla \times v $ as a new unknown. From (\ref{followingidentity}) and \eqref{VecIdent-1}, we 
can extract a {\it derived system} on $ U^\mi_\mv := (w,B^*) $, which is
 \begin{equation}\label{syssimplinet}  
\left \lbrace \begin{array}{l}
\part_t w + (v \cdot \nabla ) w + ( \mathscr K^\mi_{-1} \, B^* \cdot \nabla ) B^* = \mathscr S^{\mi w}_{\mv0} \, U^\mi_{\mv}  \, ,  \medskip \\
\displaystyle \part_t B^* +  \bigl( ( v - d \, \mathscr K^\mi_{-1} \, B^* ) \cdot \nabla \bigr) B^* + \bigl( \mathscr K^\mi_{-1} \, B^* \cdot \nabla 
\bigr) w = \mathscr S^{\mi B^*}_{\mv 0} \, U^\mi_{\mv}  \, .
\end{array} \right. 
 \end{equation}
 In (\ref{syssimplinet}), the velocity $ v $ must be deduced from $ w $ through the Biot-Savart law (\ref{Biot-Savart}), while the operator 
 $ \mathscr S^{\mi}_{\mv 0} = (\mathscr S^{\mi w}_{\mv 0},\mathscr S^{\mi B^*}_{\mv 0}) $ is given by 
$$ \begin{array}{l}
\displaystyle \mathscr S^{\mi w}_{\mv 0} \, U^\mi_{\mv}  := (B^* \cdot \nabla) \mathscr K^\mi_{-1} \, B^* + \sum_{i=1}^3 w_i \ 
\cM^\mi_i (w) \, ,  
\vspace{-4pt} \\
\displaystyle \mathscr S^{\mi B^*}_{\mv 0} \, U^\mi_{\mv} := - \, d \, (B^* \cdot \nabla ) (\mathscr K^\mi_{-1} \, B^*) + (w \cdot 
\nabla) (\mathscr K^\mi_{-1} \, B^*) + \sum_{i=1}^3 B^*_i \ \cM^\mi_i (w) \, ,
\end{array} $$
with $ \cM^\mi_i $ defined as in Lemma \ref{elldpourw}. By combining Lemmas \ref{ellde2} and \ref{elldpourw}, we obtain that 
$ \mathscr S^{\mi}_{\mv 0} $ is a (non linear) pseudo-differential operator of order zero. Hence, it can be viewed as a source term. 
Observe that $ B $ has disappeared from (\ref{syssimplinet}). There is no longer any 
need for (\ref{lienBB*2}), whereas (\ref{divfreeini}) becomes
\begin{equation}\label{divfreekeep}  
\nabla \cdot w = 0 \, , \qquad \nabla \cdot B^* = 0 \, .
\end{equation}
For $ 0 < d_e \ll 1 $, the inertial modifications appear at the level of (\ref{compressibledeb}) as perturbative terms. As such, the impact of inertial 
terms could seem to be marginal. But this is not so:
\begin{itemize}
\item [-]At high frequencies (for $ \vert \xi \vert \geq 1/d_e $), as suggested by (\ref{compressible}), the inertial contributions 
compete with the other influences.
\item [-]The constitutive relation (\ref{lienBB*inideb}) induces (through Lemma \ref{ellde2}) a complete reordering of the unknowns. The change is 
brutal from $ d_e = 0 $ to $ d_e > 0 $. Once $ d_e > 0 $, the terms which manage in standard MHD the Alfven and Magnetosonic waves are relegated 
inside the source term $ \mathscr S^{\mi}_{\mv 0} \, U_{\mv} ^\mi $, where they play the role of zero order contributions. Still, they participate to lower order dispersive 
effects.
\item [-]In XMHD, new terms become predominant. Emphasis is given to the symmetric part which, in the left part of (\ref{syssimplinet}), involves 
$ \mathscr K^\mi_{-1} $.
\end{itemize} 

\noindent In other words, the passage from (\ref{compressibledeb}) to (\ref{compressible}), and especially from (\ref{compressible}) to 
(\ref{syssimplinet}), is very singular (there is no smooth transition from $ d_e = 0 $ to $ d_e > 0 $). It makes appear the (hidden) hyperbolic 
structure of XMHD. The consequence in terms of the occurrence and organization of waves is as explained just after Theorem \ref{theoprin}.

\begin{rem} \label{Inertialobs} [Energy spectra] In \cite{ALM16,MLM17}, using a Kolmogorov-like analysis and hypotheses (regarding the energy and 
helicity cascades), the authors obtain the energy spectra of XMHD in different (ideal, Hall and inertial) regimes. This study confirms that many types 
of waves overlap in XMHD, while inertial features can overtake at high frequencies. 
\end{rem}


\subsubsection{Inertial waves} \label{The inertial waves} Excluding for the moment the coupling induced by the source terms and assuming 
that $ d_i= 0 $ (so that $ d=0$), the system (\ref{syssimplinet}) reduces to 
\begin{equation}\label{syssimplinetsecond}  
\left \lbrace \begin{array}{l}
\part_t w + (v \cdot \nabla ) w + ( \mathscr K^\mi_{-1} \, B^* \cdot \nabla ) B^* = 0 \, ,  \medskip \\
\displaystyle \part_t B^* +  ( v \cdot \nabla \bigr) B^* + \bigl( \mathscr K^\mi_{-1} \, B^* \cdot \nabla \bigr) w = 0 \, .
\end{array} \right. 
 \end{equation}
Noting $ B^0_{\pm} := B^* \pm w $, this is the same as a nonlinear coupled system of two transport equations, namely
$$ \part_t B^0_{\pm} + (v \cdot \nabla ) B^0_{\pm} \pm \frac{1}{2} \ \bigl( \mathscr K^\mi_{-1} \, (B^0_+ + B^0_-) \cdot \nabla \bigr) B^0_\pm 
= 0 \, . $$
We can immediately recognize two distinct eigenvalues (which provide a first access to  inertial waves), each of multiplicity $ 3 $, which are
\begin{equation}\label{vpkappabis}  
\lambda_\pm \equiv \lambda_\pm (v,B_+^0, B_-^0, \xi) := v \cdot \xi \pm \frac{1}{2} \ \mathscr K^\mi_{-1} \, (B^0_+ + B^0_-) \cdot \xi \, . 
 \end{equation}
These eigenvalues $ \lambda_\pm $ are formally genuinely nonlinear in the sense that
$$ ( \tilde B^0_\pm \cdot \nabla_{B^0_\pm} ) \lambda_\pm (v,B_+^0, B_-^0, \xi) = \pm \frac{1}{2} \ \mathscr K^\mi_{-1} \, \tilde B^0_\pm\ 
\cdot \xi 
\not \equiv 0 \, .  $$
It turns out that the above elementary diagonalisation procedure can be generalized to the whole system. Indeed, with
    \begin{equation}\label{vpkappaccd}  
   \qquad   B_\pm^d := B^\ast + \kappa_\pm^d \, \nabla \times v \, , \qquad v_\pm^d := v-\kappa_\mp^d \, \nabla \times B \, , \qquad 
   \kappa_\pm^d := \frac{1}{2} \ \Bigl( d \pm \sqrt{d^2 + 4} \Bigr) \, , 
       \end{equation}
    the incompressible XMHD equations \eqref{syssimpli} can be recast as
 \begin{equation}\label{bdazkh}  
 \partial_t B_\pm^d + \nabla \times (B_\pm^d \times v_\pm^d) =0 \, .
     \end{equation}    
The formulation (\ref{vpkappaccd})-(\ref{bdazkh}) is implicit in \cite{ALM16,LMM16}. From (\ref{vpkappaccd}), following the preceding lines, we can 
extract 
        \begin{equation}\label{bdazkhencore}  
 v_\pm^d = \nabla \times (-\Delta)^{-1}\bigg(\frac{B_+^d -B_-^d}{\kappa_+^d -\kappa_-^d}\bigg) - \kappa_\mp^d \ \nabla \times (1-\Delta)^{-1}
 \bigg(\frac{\kappa_+^d \, B_-^d - \kappa_-^d \, B_+^d}{\kappa_+^d -\kappa_-^d}\bigg) \, .
    \end{equation}
   In other words, incompressible XMHD can also be seen as two incompressible transport equations on $B_\pm^d$ with velocities $v_\pm^d$, 
   where the latter are given in terms of $B_\pm^d$ by the generalized Biot--Savart type laws (\ref{bdazkhencore}). The unknowns $ B_\pm^d = 
   B^\ast + \kappa_\pm^d \, w $ are made of adequate linear combinations of $ B^* $ and $ w $, together with a link to $ v $ and therefore 
   $ v_\pm $ (in order to close the system). As in (\ref{syssimplinet}), the unknowns are in fact the components of $ U^\mi_{\mv}  $. As in 
   (\ref{syssimplinet}), the system (\ref{bdazkh}) completed with (\ref{bdazkhencore}) is a quasilinear symmetric system whose both coefficients 
   and source terms take the form of zero order pseudo-differential operators. Working with (\ref{syssimplinetsecond}) or (\ref{bdazkh}) are two 
   equivalent options. In this text, we select the approach through (\ref{syssimplinetsecond}). 

\smallskip

\noindent At the level of (\ref{syssimplinet}),  in terms of polarization, the influence of $ v $ and $ d_i $ is just diagonal, while the impact of 
$ d_e $ is not. Let us now assume that $ d > 0 $. Select some special solution $ (\bar w , \bar B^*) $  to the system (\ref{syssimplinetsecond}). 
We can consider the (one order part of the) linearized equations along $ (\bar w , \bar B^*) $ associated with (\ref{syssimplinet}), which are
\begin{equation}\label{syssimplinetsecondlinearized}  
\left \lbrace \begin{array}{l}
\part_t \dot w + (\bar v \cdot \nabla ) \dot w + ( \mathscr K^\mi_{-1} \, \bar B^* \cdot \nabla ) \dot B^* = 0 \, ,  \medskip \\
\displaystyle \part_t \dot B^* +  \bigl( ( \bar v - d \, \mathscr K^\mi_{-1} \, \bar B^* ) \cdot \nabla \bigr) \dot B^* + \bigl( 
\mathscr K^\mi_{-1} \, \bar B^* \cdot \nabla \bigr) \dot w = 0 \, .
\end{array} \right. 
 \end{equation}

\begin{defi} \label{definertial} The inertial waves (related to the choice of $ \bar w $ and $ \bar B^* $) are carried by the two eigenvalues 
$ \lambda_\pm (\xi) $ which are each with multiplicity $ 3 $ of the linear hyperbolic system (\ref{syssimplinetsecondlinearized}), namely 
\begin{equation}\label{vpkappa}  
\lambda_\pm \equiv \lambda_\pm (\xi) := v \cdot \xi - \kappa_\pm^d \ (\mathscr K^\mi_{-1} \, \bar B^*) \cdot \xi \, , \qquad \kappa_\pm^d
:= \frac{1}{2} \ \big( d \pm \sqrt{d^2 + 4}\big) \, .
\end{equation}
\end{defi}

\noindent To observe experimentally inertial waves, two conditions must be fulfilled:
\begin{itemize}
\item [-] The plasma must be sufficiently energetic to trigger high frequencies $ \vert \xi \vert \geq 1/d_e $.
\item [-] The data must be expressed in terms of $ w $ and $ B^* $ (or even better $ B^d_\pm $). Indeed, information collected just in terms of 
$ v $ could be difficult to interpret.
\end{itemize}
 

\subsection{Proof of Theorem \ref{wellposednessideal}} \label{lwellposedin} We start by showing Theorem \ref{wellposednessideal}
under the more restrictive regularity assumption $ s > 7/2 $. We refer to Section \ref{potentialformulations} for the optimal result. In 
particular, at time $ t = 0 $, we know that (with $ \tilde s := s-1 $)
\begin{equation}\label{inidatasimpl}  
U^\mi_{\mv}  (0,\cdot) = U^\mi_{\mv 0} = (w_0, B^*_0) \in \cD^{\tilde s} (\RR^3;\RR^3)^3 \, , \quad w_0 := \nabla \times v_0 \, , \quad 
\tilde s > 5/2  \, .
\end{equation}
Any smooth solution to (\ref{divfreeini})-(\ref{syssimpli})-(\ref{lienBB*2})-(\ref{reginidatain}) leads to a solution to 
(\ref{syssimplinet})-(\ref{divfreekeep})-(\ref{inidatasimpl}), and conversely. We study below the time evolution of 
the $ L^2 $-norm of $ U^\mi_{\mv}  $ (assuming for the moment that $ U^\mi_{\mv}  $ is bounded in the large $ \cH^{\tilde s} $-norm).

\begin{lem} \label{L2inestimate} [$L^2 $-energy estimate for the incompressible vorticity formulation] Let $ T > 0 $.  Assume that the function 
$ U^\mi_{\mv}  \in C([0,T];\cD^{\tilde s}) $ with $ {\tilde s} > 5/2 $ is a solution to (\ref{syssimplinet}) with initial data (\ref{inidatasimpl}). 
Then, we can find a constant $ C $ depending only on the $ C([0,T];\cH^{\tilde s}) $-norm of $ U^\mi_{\mv}  $ such that
\begin{equation}\label{0-energy estimate}  
\parallel U^\mi_{\mv}  (t,\cdot) \parallel_{L^2} \leq \parallel U^\mi_{\mv 0} \parallel_{L^2} \, e^{C \, t} \, , \qquad \forall \, t \in [0,T]  \, .
\end{equation}
\end{lem} 

\begin{proof} Multiply the first and second equation of (\ref{syssimplinet}) respectively by $ w $ and $ B^* $, and then integrate 
with respect to $ x $. Since $ v \in \cD^{\tilde s} $, the contributions issued from the (transport) diagonal part involving $ v \cdot \nabla $ 
disappear. After integrations by parts, there remains 
 \begin{multline*}
   \frac{1}{2} \ \frac{d}{dt}  \bigg( \int_{\RR^3} \vert U^\mi_{\mv}  (t,\cdot)  \vert^2 \ dx  \bigg) 
    = -  \ \frac{d}{2} \, \int_{\RR^3} \nabla \cdot \bigl( \mathscr K^\mi_{-1} \, B^* \bigr) \ \vert B^* \vert^2 \ dx
    \\ + \int_{\RR^3} \nabla \cdot \bigl( \mathscr K^\mi_{-1} \, B^* \bigr) \ (w \cdot B^* ) \ dx
    +  \int_{\RR^3} U^\mi_{\mv}  \cdot \mathscr S^\mi_{\mv 0} U^\mi_{\mv}  \, dx \, .
  \end{multline*}
The Fourier multiplier $ (\mathscr L^\mi_2)^{-1} $ commutes with $ \nabla \cdot $, while $ \nabla \cdot \nabla \times \equiv 0 $. Thus
$$ \frac{1}{2} \ \frac{d}{dt} \bigg( \int_{\RR^3} \vert U^\mi_{\mv}  (t,\cdot)  \vert^2 \, dx  \bigg) = \int_{\RR^3} U^\mi_{\mv}  \cdot \mathscr S^\mi_{\mv 0}
U^\mi_{\mv}  \ dx \, .$$
Below, we use the Sobolev embedding theorem $ \cH^{\tilde s} \hookrightarrow L^\infty $ (knowing that $ \tilde s > 5/2 $). We exploit the condition 
$ \nabla \cdot v = 0 $ to deal with the sum of products $ w_i \ \cM^\mi_i (w) $. We also implement Lemmas \ref{ellde2} and \ref{elldpourw} to get 
$$ \begin{array}{rl}
\displaystyle \bigg \vert \int_{\RR^3} w \cdot \mathscr S^{\mi w}_{\mv 0} U^\mi_{\mv}  \, dx \bigg \vert \leq \! \! \! & \displaystyle  
\parallel w \parallel_{L^\infty} \ \sum_{i=1}^3 \bigl( \parallel B^*_i \parallel_{L^2} \ \parallel \part_i  \mathscr K^\mi_{-1} \, B^*\parallel_{L^2} + 
\parallel w_i \parallel_{L^2} \ \parallel \cM^\mi_i (w) \parallel_{L^2} \bigr) \smallskip \\
\lesssim  \! \! \! & \displaystyle \parallel U^\mi_{\mv}  \parallel_{C([0,T];\cH^{\tilde s})} \ \parallel U^\mi_{\mv}  \parallel_{L^2}^2 \, ,
\end{array} $$
as well as 
$$ \begin{array}{rl}
\displaystyle \bigg \vert \int_{\RR^3} B^* \cdot \mathscr S^{\mi B^*}_{\mv 0} U^\mi_{\mv}  \, dx \bigg  \vert \leq \! \! \! & \displaystyle d \ \parallel B^* 
\parallel_{L^\infty} \sum_{i=1}^3 \parallel B^*_i \parallel_{L^2} \, \parallel \part_i \mathscr K^\mi_{-1} \, B^* \parallel_{L^2} \smallskip \\
\ & \displaystyle + \, \parallel B^* \parallel_{L^\infty} \sum_{i=1}^3 \bigl( \parallel w_i \parallel_{L^2} \, \parallel \part_i \mathscr K^\mi_{-1} \, B^* 
\parallel_{L^2} + \parallel B^*_i \parallel_{L^2} \, \parallel \cM^\mi_i (w) \parallel_{L^2} \bigr) \smallskip \\
\lesssim \! \! \! & \displaystyle \parallel U^\mi_{\mv}  \parallel_{C([0,T];\cH^{\tilde s})} \ \parallel U^\mi_{\mv}  \parallel_{L^2}^2 \, . 
\end{array} $$
By Gr\"onwall's inequality, we recover (\ref{0-energy estimate}).
\end{proof}

\noindent The proof of Lemma \ref{L2inestimate} serves to confirm that the source term is indeed of order $ 0 $. To go further, we have 
to write down a scheme \cite{Alin,Maj84} in order to use a fixed-point method. To this end, we need to implement the linearized version 
of (\ref{syssimplinet}). Then, we have to perform energy estimates in order to obtain a control in the large norm $ L^\infty ([0,T];\cH^{\tilde s}) $, 
and a convergence in the small norm $ L^\infty ([0,T];L^2) $. 

\noindent When doing this, the coefficients (which are transparent in the above proof) are implied. The only difficulty could come from 
the operator $ (\mathscr K^\mi_{-1} \, B^*) \cdot \nabla $ but the coefficient $ \mathscr K^\mi_{-1} \, B^* $ 
is of order $ 0 $ (and even of order $ -1 $) as required. Thus, $ L^2 $-energy estimates are available for the linearized equations along 
the same lines as above. 

\noindent To get $ \cH^{\tilde s} $-bounds, we have to commute the linearized equation with spatial derivatives $ \part^\alpha_x $ with 
$ \vert \alpha \vert \leq \tilde s $, and exploit linear estimates of nonlinear functions. This falls under the scope of the general strategy 
\cite{Alin,Maj84} to solve quasilinear symmetric systems. The details, which are standar and long, are not reproduced here. 
The conclusion is that the Cauchy problem associated with (\ref{syssimplinet}) is well-posed in $ \cH^{\tilde s} $ for $ \tilde s > 5/2 $.

\noindent From the $ \cH^{\tilde s} $-solutions to (\ref{syssimplinet}) with $ \tilde s > 5/2 $, we recover solutions to (\ref{syssimpli}), 
which are such that $ (v,B^*) (t,\cdot) \in \cH^s \times \cH^{s-1} $ with $ s := \tilde s +1 > 7/2 $. Moreover, from Lemma \ref{ellde2} together 
with (\ref{ell-1}), we obtain that $ B (t,\cdot) \in \cH^{s+1} $. This concludes the proof of Theorem \ref{wellposednessideal} at least on 
condition that $ s > 7/2 $.

\begin{rem} \label{versus} [Propagated $ L^2 $-energy for (\ref{syssimplinet}) versus conserved quantity for (\ref{syssimpli})] From Lemma 
\ref{ellde2}, we know that $ \parallel \nabla \times B \parallel_{L^2} = \parallel \mathscr K^\mi_{-1} \, B^* \parallel_{L^2} \lesssim \parallel 
B^* \parallel_{L^2} $.
It is clear that, with $ \cE^\mi $ as in (\ref{conservedprop}), we have $ \cE^\mi \lesssim \parallel U^\mi_{\mv}  \parallel_{L^2} $. 
The opposite is false. In other words, Lemma \ref{L2inestimate} is not a corollary of Lemma \ref{energypreserin}.
\end{rem}


\section{The compressible framework} \label{vorcompressibleframework} In this section, $ p: \RR_+ \rightarrow \RR $ is a given strictly 
increasing smooth function of $ \rho $. We extend here (\ref{reginidatain}) by putting aside the condition $ \nabla \cdot B_0^* = 0 $. As a 
matter of fact, we consider general vector fields $ B^* $. This is made possible by the following remark.

\begin{lem} \label{preservationB*} [Conservation of the magnetic divergence] Any solution to (\ref{compressible})-(\ref{lienBB*ini})-(\ref{inidata}) 
is such that 
\begin{equation}\label{divfreepreser}  
 \nabla \cdot B^* = \nabla \cdot B = \nabla \cdot B^*_0 \, .
 \end{equation}
\end{lem} 

\begin{proof} This is just because $ \part_t (\nabla \cdot B^*) = 0 $. 
\end{proof}

\noindent The whole vector field $ B^* $ (resp. $ B $) can be reconstituted from $  \nabla \cdot B^* $ and $ \nabla \times B^* $ (resp. 
from $ \nabla \cdot B $ and $ \nabla \times B $) by solving the div-curl system (see Subsection \ref{div-curlsystem}). The parts 
$ \nabla \cdot B^* $ and $ \nabla \cdot B $ are determined by (\ref{divfreepreser}). In particular, with $ \mathscr P = P(D_x) $ where 
$ P $ is as in (\ref{l2projector}), retain that
\begin{equation}\label{retainthat}   
B^* = \mathscr Q \, B_0^* + \mathscr P \, B^* \, , \qquad \mathscr Q = \id - \mathscr P \, . 
 \end{equation}
In other words, replacing everywhere $ B^* $ as indicated above, the system (\ref{compressible}) reduces to an equation on 
$ (\rho,v,\mathscr P B^*) $, while the constitutive relation (\ref{lienBB*ini}) is aimed to deduce  $ \mathscr P B $ from $ \mathscr P B^* $, 
or equivalently $ \nabla \times B $ from $ \nabla \times B^* $. 

\smallskip

\noindent Lemma \ref{preservationB*} is straightforward. It is however highlighted because it plays a crucial role for the 
reason explained in the remark below.

\begin{rem} \label{yeti} [A consequence of the foliation by vector fields having a fixed divergence] Let $ C $ be a smooth vector field viewed 
as a coefficient. From (\ref{followingidentity}), we have the decomposition 
\begin{equation}\label{decompoooo}  
\mathscr T_C \, B^* \equiv \mathscr T \, B^* := \nabla \times (C \times B^*) = \mathscr T_1 \, B^* + \mathscr T_0 \, B^*  
 \end{equation}
with
\begin{equation}\label{definitionuf}  
\mathscr T_1 \, B^* := (\nabla \cdot B^*) \ C - (C \cdot \nabla) B^* \, , \qquad \mathscr T_0 \, B^* := (B^* \cdot \nabla) C -  (\nabla \cdot C) \ B^* \, . 
 \end{equation}
The operator $ \mathscr T_0 $ is of order $ 0 $, while $ \mathscr T_1 $ is of order $ 1 $. The action of $ \mathscr T_1 $ is not skew-adjoint. 
As such, it is not compatible with energy estimates. However, knowing (\ref{divfreepreser}), we should opt for $ \mathscr T \, B^* = \tilde 
{\mathscr T}_1 \, B^* + \tilde {\mathscr T}_0 \, B^* $ with
\begin{equation}\label{definitionufbis}  
\tilde {\mathscr T}_1 \, B^* := - (C \cdot \nabla) B^* \, , \qquad \tilde {\mathscr T}_0 \, B^* := (\nabla \cdot B^*_0) \ C + (B^* \cdot \nabla) C -  
(\nabla \cdot C) \ B^* \, . 
 \end{equation}
The operator $ \tilde {\mathscr T}_1 $ is skew-adjoint. Contrary to $ {\mathscr T}_1 $, it can be dealt with in the energy estimates without losses 
of derivatives. This trick will be repeatedly used. As a matter of fact, we will systematically replace $ \nabla \cdot B^* $ by $ \nabla \cdot B^*_0 $. 
 \end{rem}

\noindent In order to make the transition from $ \nabla \times B^* $ to $ \nabla \times B $, we have to exploit conveniently the constitutive relation (\ref{lienBB*ini}).  
To this end, we follow a plan similar to Section \ref{vorcasincompressible}. In Subsection \ref{transformationcompressible}, we come back to the 
content of (\ref{lienBB*ini}) but this time when $ \rho $ is a non constant function. In Subsection \ref{transbiscompressible}, we adapt to the compressible 
context the change of variables of Subsection \ref{transbis}. In Subsection \ref{lwellposedcomp}, we derive energy estimates to show Theorem 
\ref{theoprin} (for $ s > 7/2 $).  At each stage, in comparison with Section \ref{vorcasincompressible}, we need to implement important and difficult modifications. 


\subsection{The compressible constitutive relation} \label{transformationcompressible} The difficulty here is to exploit (\ref{lienBB*ini}) in 
order to express $ \mathscr P B $ in terms of $ \mathscr P B^* $. In this subsection, we fix a time $ t \in \RR_+ $, and we assume that the 
function $ \rho (t,\cdot) : \RR^3 \rightarrow \RR$ is bounded and positive. More precisely, we impose
\begin{equation}\label{minorationrho}  
\exists \, (c,C) \in \RR^2 \, ; \qquad 0 < c \leq \rho(t,x) \leq C \, , \qquad \forall \, x \in \RR^3  \, .
\end{equation}
We also suppose that the function $ \rho (t,\cdot) $ is smooth enough, say in $ \cH^s (\RR^3) $ with $ s > 7/2 $. By this way, we can use the 
pseudo-differential calculus with coefficients in $ \cH^s $, as developed for instance in \cite{Marschall, Tay91}. In what follows, we will sometimes
omit to mention the presence of $ t $. From (\ref{lienBB*ini}), we get that
\begin{equation}\label{lienBB*inire}  
 \nabla \times B^* = \mathscr  L^{\mc}_2 \, \bigg( \frac{\nabla \times B}{\rho(x)} \bigg) \, , \qquad \mathscr  L^{\mc}_2 := \rho(x) \, \id + \nabla \times 
\nabla \times \, . 
\end{equation}
We look at $ \mathscr  L^{\mc}_2  : L^2 (\RR^3;\RR^3) \rightarrow L^2 (\RR^3;\RR^3) $ as an unbounded operator \cite{CheRay}. 

\begin{lem} \label{zeroorder} [Inverse of $ \mathscr  L^{\mc}_2 $] The operator $ (\mathscr  L^{\mc}_2)^{-1} : L^2 \rightarrow L^2 $ is well-defined
and bounded. 
\end{lem} 

\begin{proof} Observe that $ \mathscr  L^{\mc}_2 $ is symmetric and positive since
\begin{equation}\label{elliordre0}  
\qquad \begin{array}{rl}
\displaystyle \int (\mathscr  L^{\mc}_2 u) (x) \cdot \bar u(x) \ dx = \! \! \! & \displaystyle  \int \rho(x) \ \vert u(x) \vert^2 \ dx  \\
\ & \displaystyle + \int \vert \nabla \times  u(x) \vert^2 \ dx 
\geq c \, \parallel u \parallel_{L^2}^2 \, , \qquad \forall u \in \cH^2 \, .  
\end{array} 
\end{equation}
The polar decomposition furnishes the existence of a densely defined, closed and self-adjoint operator $ \mathscr X^{\mc} : 
L^2 \rightarrow L^2 $ with domain $ \text{Dom} \, (\mathscr X^{\mc}) $ such that $ \mathscr  L^{\mc}_2 = (\mathscr X^{\mc})^* \, \mathscr X^{\mc} $, 
and therefore
$$ \begin{array}{ll}
\parallel \mathscr X^{\mc} \, u \parallel_{L^2} \geq \sqrt c \ \parallel u \parallel_{L^2} \, , \quad & \forall \, u \in \text{Dom} \, (\mathscr X^{\mc}) \, , \medskip\\
\parallel (\mathscr X^{\mc})^* \, u \parallel_{L^2} \geq \sqrt c \ \parallel u \parallel_{L^2} \, , \quad & \forall \, u \in \text{Dom} \, \bigl( (\mathscr X^{\mc})^*\bigr) \, .
\end{array} $$
Starting from there, the two operators $ \mathscr X^{\mc} $ and $ (\mathscr X^{\mc})^* $ are invertible (Theorem 3.3.2 in \cite{Dere}, or see also \cite{CheRay}). 
The same applies to $ \mathscr  L^{\mc}_2 $ with $ (\mathscr L^{\mc}_2)^{-1} = (\mathscr X^{\mc})^{-1} \circ \bigl( (\mathscr X^{\mc})^* \bigr)^{-1} $. 
\end{proof}

\noindent For smooth enough vector fields $ B^* $, the relation (\ref{lienBB*inire}) amounts to the same thing as 
\begin{equation}\label{inverseconstitutive}  
\frac{\nabla \times B}{\rho(x)} = \mathscr K^{\mc}_{-1} \, B^* \, , \qquad \mathscr K^{\mc}_{-1} := (\mathscr  L^{\mc}_2)^{-1} \, \nabla \times \equiv \mathscr 
K^{\mc}_{-1} \mathscr P \, .
\end{equation}
The system (\ref{compressible}) where $ \nabla \times B / \rho $ is replaced everywhere as indicated in (\ref{inverseconstitutive}) is enough 
to recover a self-contained system on $ {}^t (\rho,v,\mathscr P B^*) $. We can progress without introducing $ B $ and without imposing $ \nabla \cdot B^* = 0 $. 
Neither (\ref{lienBB*ini}) nor (\ref{divfreeini}) are needed. It suffices to rely on (\ref{inverseconstitutive}). Still, the passage through (\ref{lienBB*ini}) and 
(\ref{divfreeini}), which is prescribed by physicists, is meaningful. First, it is a way to deduce the final constitutive relation (\ref{inverseconstitutive}). 
Secondly, it is more adapted in view of the potential formulation (in Section \ref{potentialformulations}). Now, one important key in continuity with Lemma 
\ref{ellde2} is to show that the restriction of $ (\mathscr L^{\mc}_2)^{-1} $ to $ \cD^r $ (for well-chosen indices $ r $) still gives rise to a gain of two derivatives. 
In other words, we have to justify the subscript $ -1 $ in $ \mathscr K^{\mc}_{-1} $.

\begin{prop} [A property of ellipticity when going from $ \nabla \times B^* $ to $ \nabla \times B / \rho $ through the constitutive relation 
(\ref{inverseconstitutive})] \label{ellllli} The action of $ (\mathscr  L^{\mc}_2)^{-1} $ is associated with a matrix valued operator whose all coefficients 
are pseudo-differential operators. Its restriction to solenoidal vector fields is elliptic of order less or equal to $ -2 $. More precisely, for $ r \in [s-2,s] $,
the action $ (\mathscr  L^{\mc}_2)^{-1} : \cD^r \rightarrow \cH^{r+2} $ is well-defined and continuous. 
\end{prop}

\noindent In comparison with Lemma \ref{ellde2}, the variations of the function $ \rho $ induce modifications:
\begin{itemize}
\item [{\it - d1.}] First, unlike $ (\mathscr L^\mi_2)^{-1} $, the image of $ (\mathscr  L^{\mc}_2)^{-1} $ on $ \cD^r $ is not $ \cD^{r+2} $. Indeed, since $ \rho $
is not constant, the action of $ (\mathscr  L^{\mc}_2)^{-1} $ implies a deformation out of the set of solenoidal vector fields. This is restored as indicated 
in (\ref{inverseconstitutive}) after multiplication by $ \rho $. The transition from $ \nabla \times B^* $ to $ \nabla \times B $ through (\ref{inverseconstitutive}) 
is not diagonal; it is not so simple. In particular, the identity
\begin{equation}\label{rolelll}  
\nabla \times B = \rho(x) \  (\tilde {\mathscr  L}^{\mc}_2)^{-1}  (\nabla \times B^*) \, , \qquad \tilde {\mathscr  L}^{\mc}_2 := \bigl( \rho(x) - \Delta \bigr) \ \id_{3 \times 3}  \, ,
\end{equation}
which could appear as the correct extrapolation of (\ref{lienBB*2}) is false. 
\item [{\it - d2.}] Secondly, there are restrictions on $ r $ which are absent (on $ s $) at the level of (\ref{differentialoperator}). on the one hand, the upper 
bound $ r \leq s $ comes from the limited regularity of $ \rho $. On the other hand, the lower bound $ s-2 \leq r $ is issued from the rules of composition 
in Sobolev spaces (in view of a para-differential calculus). We can further illustrate these two conditions  by looking at the elliptic equation $ \bigl( \rho(x) 
- \Delta \bigr) \, u = f $ where $ u \in \cH^{s-1} $ and $ f \in \cH^r $ with $ r \in [s-2,s] $. Then, knowing that $ s > 7/2 $, from Theorem  1.2.A in \cite{Tay91}, 
we obtain that $ u \in \cH^{r+2} $. 
\end{itemize}

\begin{proof} The rest of Subsection \ref{transformationcompressible} is devoted to the proof of Proposition \ref{ellllli}. To overcome the difficulty {\it d1}, 
we must keep track of derivative losses concerning $ \rho $. To solve {\it d2}, we follow a procedure in three steps:
\begin{itemize}
\item [-] In Paragraph \ref{paratory work}, we explain our method of unitary conjugation, and we introduce preliminary tools like the Weyl quantization.
\item [-] In Paragraph \ref{Unitaryreduction}, we show by Weyl calculus that $ \mathscr  L^{\mc}_2 $ is (almost) unitary equivalent to a block diagonal 
action. In fact, the principal symbol of $ \mathscr  L^{\mc}_2 $ has two distinct eigenvalues: $ \rho(x) $ and  $ \rho(x) + \vert \xi \vert^2 $ which are 
respectively of multiplicity one and two. The unitary reduction reveals a $ 2 \times 2 $ elliptic block of order $ 2 $, corresponding to the second eigenvalue 
and involving (after inversion) a gain of two derivatives. The difficulty is to show that this gain remains effective on $ \cD^r $, while it could be destroyed by 
the variations of $ \rho $. The presence of a non constant function $ \rho $ produces nonzero commutators and by this way non diagonal terms. In this line, 
note again that the relation (\ref{rolelll}) is not verified.
\item [-]In Paragraph \ref{parametrix}, to remedy this, we construct an approximate parametrix, and we check that its properties allow to conclude. 
\end{itemize} 
\vskip -3mm
\end{proof}


\subsubsection{Preparatory work}\label{paratory work} We denote by $OP\cH^s S^m$ the set of pseudo-differential operators of order less or equal 
to $m$ with symbols in $\cH^s$ (e.g., see \cite{Marschall,Tay91}), and simply ${\rm Op}(m)$ an element of $OP\cH^r S^m$ (for some unspecified 
$ r = s-1 $ or $ r = s $). In view of (\ref{elliordre0}), the action of $ \mathscr  L^{\mc}_2 $ is (at least) elliptic of order $ 0 $. Thus, to evaluate its precise 
order, it suffices to consider what happens for large frequencies, that is for $ \xi $ with $ \vert \xi \vert \gg 1 $. The action of $ \mathscr  L^{\mc}_2 $ is 
achieved through a matrix valued differential operator, which is non diagonal. In line with (\ref{defdePxi}) and (\ref{fourierac}), a first attempt to obtain a 
block diagonal form is to look at 
$$ \mathscr O_0^{-1} \, \mathscr  L^{\mc}_2 \, \mathscr O_0 = \mathscr D^{\mc} + \mathscr E \, , \qquad \mathscr E := \mathscr O_0^{-1} \, \bigl \lbrack 
\rho(x) \,  \id , \mathscr O_0 \bigr \rbrack \, , $$
where $ \mathscr D^{\mc} $ and $ \mathscr E = \mathscr E^* $ are given by
$$ \qquad \mathscr D^{\mc} := \left( \begin{array}{ccc}
\rho(x) & 0 & 0 \\
0 & \rho(x) - \Delta & 0 \\
0 & 0 & \rho(x) - \Delta 
\end{array} \right) \, , \qquad \mathscr E =\left( \begin{array}{ccc}
\mathscr E_{11} & \mathscr E_{12} & \mathscr E_{13} \\
\mathscr E^*_{12} & \mathscr E_{22} & \mathscr E_{23} \\
\mathscr E^*_{13} & \mathscr E^*_{23} & \mathscr E_{33}
\end{array} \right) \, . $$
It is clear that $ \mathscr D^{\mc}  \in OP\cH^sS^{2} $. Thus, in the absence of $ \mathscr E $, Proposition \ref{ellllli} would be a direct consequence 
of Theorem  1.2.A in \cite{Tay91}. There remains to explain how to absorb the above remainder $ \mathscr E $.

\smallskip

\noindent Denoting by $ \mathscr O_{0ij} = O_{0ij} (D_x) $ the elements of the matrix valued pseudo-differential 
operator $ \mathscr O_0 $ (which are all of order $ 0 $), we find that (e.g., see Corollary~4.1 in \cite{Alin})
$$ (\mathscr E)_{ij} = \mathscr O_0^{-1} \, \Bigl( \bigl \lbrack \rho(x) , \mathscr O_{0ij} \bigr \rbrack \Bigr)_{ij} = \mathscr O_0^{-1} \, \Bigl( {\rm Op} \bigl( i \, \{ 
O_{0ij} (\xi) , \rho(x) \} \bigr) \Bigr)_{ij} + \rm Op(-2) \, ,$$
where we have introduced the Poisson bracket, which is given in the phase space $(x,\xi)$ by
$$ \{ f ,g  \} :=\sum_{i=1}^3 \partial_{\xi_i} f \ \partial_{x_i}g-\sum_{i=1}^3 \partial_{x_i}f \ \partial_{\xi_i}g\, . $$ 
We see on this formula that $ \mathscr E \in OP\cH^{s-1} S^{-1} $.
The above reduction is not yet sufficient in order to conclude (due to the presence of 
non zero coefficients $ \mathscr E_{1\star} $). To (partially) further absorb $ \mathscr E $, the idea is to find a unitary operator $ \mathscr V $ such that
 \begin{equation}\label{objectif}  
\qquad \mathscr V^* \, (\mathscr D^{\mc} + \mathscr E) \, \mathscr V = \mathscr D^{\mc} + \mathscr E_r +  {\rm Op}(-4) \, , \quad \mathscr E_r := \left( \begin{array}{ccc}
\mathscr E_{11} & 0 & 0 \\
0 & \mathscr E_{22} & \mathscr E_{23} \\
0 & \mathscr E^*_{23} & \mathscr E_{33}
\end{array} \right) = \mathscr E_r^* \, .
\end{equation}
To this end, we seek $ \mathscr V $ in the form $ \mathscr V= e^{i \mathscr A} $ where $ \mathscr A $ is a self-adjoint pseudo-differential operator with real valued 
symbol $ A $. When doing this, to facilitate calculations, it is more appropriate to work with the Weyl quantization (with symbol $ A $) 
given by
$$  \mathscr A \, u (x) \equiv {\rm Op}^W (A) u(x) := \frac{1}{(2 \pi)^3} \, \int_{\RR^3} \! \int_{\RR^3} e^{i (x-y)\cdot \xi} \ A \Bigl( \frac{x+y}{2}, \xi \Bigr) \ 
u(y) \ dy \, d \xi \, , \quad u \in \cD(\RR^3 ) \, . $$
We recall that any operator $ \mathscr A $ having a real symbol $ A $ is self-adjoint, so that $ e^{i \mathscr A} $ is a unitary pseudo-differential 
operator satisfying
$$ e^{i \mathscr A} = \sum_{k=0}^{+\infty} \, \frac{i^k}{k!} \ \mathscr A^k =  \id + i \, \mathscr A + \cdots \, , \qquad (e^{i \mathscr A} )^* = e^{-i \mathscr A} \, . $$
When $ \mathscr A $ is of negative order ($ m < 0 $), the above sum implements terms $ \mathscr A^k $ which are of decreasing orders $ k \, m $. For instance, 
the above remainder (marked by $ \cdots $) is in $ {\rm Op} (2 \, m) $.


\subsubsection{Unitary reduction}\label{Unitaryreduction} Assume that $ \mathscr A = (\mathscr A_{ij})_{ij} $ is a self-adjoint operator of negative order $ -3 $. 
Then, we deal with
$$ \begin{array}{rl}
\mathscr V^* \, (\mathscr D^{\mc} + \mathscr E) \, \mathscr V \! \! \! & = \bigl( \id - i \mathscr A +  {\rm Op}(-6) \bigr) \ (\mathscr D^{\mc} + \mathscr E) \ \bigl(  \id + i 
\mathscr A +  {\rm Op}(-6) \bigr) \medskip \\
\ & = \mathscr D^{\mc} + \mathscr E + i \ \lbrack \mathscr D^{\mc} , \mathscr A \rbrack + i \ \lbrack \mathscr E , \mathscr A \rbrack + \mathscr A \, (\mathscr D^{\mc} + 
\mathscr E) \, \mathscr A + {\rm Op} (-4) \medskip\\
\ & = \mathscr D^{\mc} + \mathscr E + i \ \lbrack \mathscr D^{\mc} , \mathscr A \rbrack + {\rm Op} (-4) \, . 
\end{array}  $$
Thus, to recover (\ref{objectif}), we have to consider the homological equation 
 $$ i \ \lbrack \mathscr D^{\mc} , \mathscr A \rbrack = \left( \begin{array}{ccc}
0 & - \mathscr E_{12} & - \mathscr E_{13} \\
- \mathscr E^*_{12} & 0 & 0 \\
- \mathscr E^*_{13} & 0 & 0
\end{array} \right) + {\rm Op} (-4) \, , $$
where the $ \mathscr E_{1j} $ with $ j\in\{2,3\} $ are given and $ \mathscr A $ is the unknown. We impose $ \mathscr A_{ij} = 0 $ for $ (i,j) $ not equal to $ (1,2) $ 
or $ (1,3) $. Then, we have to solve
$$ i \ \bigl \lbrack \rho(x) , \mathscr A_{1j} \bigr \rbrack +  i \ \mathscr A_{1j} \ \Delta = - \mathscr E_{1j} + {\rm Op} (-4) \, . $$
Assuming that $ \mathscr A $ is in $ {\rm Op}(-3) $, the commutator $ \lbrack \rho , \mathscr A_{1j} \rbrack $ is in $ {\rm Op}(-4) $, and this reduces to
$$ \mathscr A_{1j} =  i \ \mathscr E_{1j} \ \Delta^{-1} + {\rm Op} (-6) \, . $$
We just take $ \mathscr A_{1j} := i \ \mathscr E_{1j} \ \Delta^{-1} $. With this choice, as required initially, the operator $ \mathscr A $ is indeed 
in $ OP \cH^{s-1} S^{-3} $. Moreover, by construction, we have access to (\ref{objectif}).

\smallskip

\noindent Retain that $ \mathscr A $ is of small order for two reasons. First, $ \mathscr E $ is obtained by commuting $ \mathscr O_0 $ 
with the diagonal matrix $ \rho(x) \, \id $ with a corresponding gain of one derivative, so that $ \mathscr E \equiv \mathscr E_{-1} $. 
Secondly, the difference $ | \xi |^2 $ between the two eigenvalues $ \rho(x) $ and $ \rho(x) + | \xi |^2 $ is of order two. After division, 
this yields a supplementary gain of two derivatives.

\smallskip

\noindent Note also that the above process does not allow to go further in the diagonalization process, in order to get rid of $ \mathscr E_{23} $. 
Indeed, the eigenvalue $ \rho(x) + | \xi |^2 $ is of multiplicity $ 2 $. As a consequence, we cannot exploit any gap between the 
(same two) eigenvalues related to the bottom $ 2 \times 2 $ block.


\subsubsection{The approximate parametrix}\label{parametrix} Consider the content of $ \mathscr D^{\mc} + \mathscr E_r $ where the orders are 
clearly separated:
\begin{itemize}
\item {\it Top $ 1 \times 1 $ block}. The scalar pseudo-differential operator $ \rho(x) + \mathscr E_{11} $ is self-adjoint. Moreover, it is elliptic of 
order $ 0 $ since its principal symbol is the function $ \rho(x) $, which satisfies (\ref{minorationrho}). Thus, it can be inverted, and its 
inverse is a self-adjoint pseudo-differential operator of order $ 0 $.
\item {\it Bottom $ 2 \times 2 $ block}. This is
$$ (\mathscr D^{\mc} + \mathscr E_r )^{Bot}_{2 \times 2} := \bigl( \rho(x)  - \Delta \bigr) \ \id_{2\times 2} + \mathscr E^{Bot}_{2 \times 2} \, , \qquad
\mathscr E^{Bot}_{2 \times 2} := \left( \begin{array}{cc}
\mathscr E_{22} & \mathscr E_{23} \\
\mathscr E^*_{23} & \mathscr E_{33}
\end{array} \right) \, , $$
where by construction $ \mathscr E^{Bot}_{2 \times 2} \in OP\cH^{s-1} S^{-1} $ acts continously on $ \cH^r $. The above operator is self-adjoint. 
Once $ \rho(x) $ satisfies (\ref{minorationrho}), it is extracted from an operator which is elliptic of order $ 0 $. As such, it is an elliptic operator of 
order $ 0 $. For large frequencies, it is (in view of its principal symbol $ \vert \xi \vert^2 $) elliptic of order $ 2 $. Consider the elliptic equation
$$ (\mathscr D^{\mc} + \mathscr E_r )^{Bot}_{2 \times 2} \, u = f \in \cH^r \, , \qquad s-2 \leq r \leq s \, , $$
or alternatively
$$ \bigl( \rho(x)  - \Delta \bigr) \, u = f - \mathscr E^{Bot}_{2 \times 2} \, u \in \cH^r \, . $$
By applying Theorem  1.2.A in \cite{Tay91}, we recover that $ u \in \cH^{r+2} $ as required. In conclusion, the operator 
$ (\mathscr D^{\mc} + \mathscr E_r )^{Bot}_{2 \times 2} $ is elliptic of order $ 2 $ on the whole phase space. It can therefore be inverted, and its inverse is 
a self-adjoint matrix valued pseudo-differential operator of  order $ -2 $. 
\end{itemize}

\noindent By construction, we have
 $$ (\mathscr  L^{\mc}_2)^{-1} = \mathscr O_0 \, \mathscr V  \, ( \mathscr D^{\mc} + \mathscr E_r - \mathscr R )^{-1} \, \mathscr V^* \, \mathscr O_0^{-1} \, , 
 \qquad \mathscr R \in OP\cH^{s-1} S^{-4} \, . $$
 On the other hand, for large frequencies, we can write
 $$ \begin{array}{rl}
 \displaystyle ( \mathscr D^{\mc} + \mathscr E_r - \mathscr R )^{-1} \! \! \! & \displaystyle = ( \mathscr D^{\mc} + \mathscr E_r )^{-1} + \sum_{k=1}^{+\infty} 
 \bigl( ( \mathscr D^{\mc} + \mathscr E_r )^{-1} \, \mathscr R \bigr)^k \, ( \mathscr D^{\mc} + \mathscr E_r )^{-1} \\
 \ & = \displaystyle ( \mathscr D^{\mc} + \mathscr E_r )^{-1} + {\rm Op} (-4) \, , 
 \end{array} $$
 and consequently
$$ (\mathscr  L^{\mc}_2)^{-1} =  \mathscr O_0 \, \mathscr V  \, \bigl( \mathscr D^{\mc} + \mathscr E_r \bigr)^{-1} \, \mathscr V^* \, \mathscr O_0^{-1} + 
{\rm Op}(-4) \, .  $$
 But on the other hand $ \mathscr V = \id + \rm Op(-3) $. Thus, the non diagonal terms induced by the actions of $ \mathscr V $ and $ \mathscr V^* $ 
 can be incorporated in a remainder. More precisely
 $$ (\mathscr  L^{\mc}_2)^{-1} = \mathscr O_0 \left( \begin{array}{cc}
\bigl( \rho(x) + \mathscr E_{11} \bigr)^{-1} & 0 \\
0 & (\mathscr D^{\mc} + \mathscr E_r )_{22}^{-1}
\end{array} \right) \mathscr O_0^{-1} + {\rm Op}(-3) \, . $$
Now, let $ v \in \cD^r $. Thus, we have $ \mathscr O_0^{-1} \, v = {}^t (0,v_2,v_3) $ with $ v_j \in \cH^r $, so that
$$ (\mathscr  L^{\mc}_2)^{-1} \, v = \mathscr O_0 \,  \left( \begin{array}{c}
0 \\
\bigl( (\mathscr D^{\mc} + \mathscr E_r )^{Bot}_{2 \times 2} \bigr)^{-1} \left( \begin{array}{c}
v_2 \\
v_3
\end{array} \right) \end{array} \right) + {\rm Op}(-3) \left( \begin{array}{c}
0 \\
v_2 \\
v_3
\end{array} \right) = {\rm Op}(-2) \left( \begin{array}{c}
0 \\
v_2 \\
v_3
\end{array} \right) , $$
which leads to the expected conclusion. 


\subsection{Transformation of the compressible  equations} \label{transbiscompressible} We start by recalling what Lemma \ref{energypreserin}
becomes in the compressible case.

\begin{lem} \label{Invariant quantities} [A decreasing enegy] Let $ U (\rho) $ be the internal energy function of the system. It must satisfy 
$ U' (\rho) = \rho^{-2} \, p (\rho) \geq 0 $, and it can be adjusted such that $ U(0)=0 $ so that $ U (\rho) \geq 0 $. In particular, for a polytropic 
equation of state, we find $ U (\rho) = c \rho^{\gamma-1} / (\gamma -1) $ where $ c $ is a positive constant and $ \gamma > 1 $ is the heat 
capacity ratio. Retain that
\begin{equation}\label{compressibledecene} 
 \cE^{\mc} (t) := \frac 1 2 \int_{\RR^3} \left (  \rho \, \vert v \vert^2 + 2 \ \rho \,  U (\rho) + B \cdot B^* \right ) (t,\cdot) \ dx \leq \cE^{\mc} (0) \, ,
\end{equation}
where 
$$ \int_{\RR^3} B \cdot B^* (t,\cdot) \, dx = \int_{\RR^3} \left ( \vert B \vert^2 + \frac{\vert \nabla \times B \vert^2}{\rho} \right ) (t,\cdot) \
dx \, .  $$ 
\end{lem} 

\begin{proof} It is well known, see especially \cite{KimuM} but also \cite{AY16,LMM16}, that XMHD (with $ \nu = 0 $) has a Hamiltonian structure 
conserving the energy $ \cE^{\mc} $. The inequality inside (\ref{compressibledecene}) comes from the dissipative effects which are  induced by the 
(fluid) bulk viscosity. Note that there are also (when $ \nu = 0 $) three independent Casimirs, see (52), (53) and (54) in \cite{AKY15}.
\end{proof} 

\noindent Since $ p' > 0 $, instead of working with $ \rho $, we can alternatively deal with  
$$ q := g(\rho) := \int_{\bar \rho}^\rho \frac{\sqrt{p'(s)}}{s} \ ds \, , \qquad g' (\rho) = \frac{\sqrt{p'(\rho)}}{\rho} > 0 \, , \qquad a(q) := 
g^{-1} (q) \ g' \circ g^{-1} (q) \, . $$
Expressed in terms of $ {}^t (q,v,B^*) $, the system (\ref{compressible}) becomes symmetric with respect to the two first lines below, 
that is especially with respect to $ (q,v) $.  We consider
\begin{equation}\label{compressibleb}   
\left \lbrace \begin{array}{l}
\part_t q + (v \cdot \nabla) q + a(q) \ \nabla \cdot v = 0 \, , \vspace{5pt}\\
\displaystyle \part_t v + (v \cdot \nabla ) v + a(q) \ \nabla q + B^* \times \mathscr K^{\mc}_{-1} \, B^* + \frac{1}{2} \ \nabla \vert 
\mathscr K^{\mc}_{-1} \, B^* \vert^2 = \nu \ \nabla (\nabla \cdot v) \, , \vspace{5pt}\\
\displaystyle \part_t B^* + \nabla \times \bigl( B^* \times (v - d \ \mathscr K^{\mc}_{-1} \, B^* ) \bigr) + \nabla \times \bigl( 
(\nabla \times v) \times \mathscr K^{\mc}_{-1} \, B^* \bigr) = 0 \, .
\end{array} \right. 
\end{equation}
In (\ref{compressibleb}), in comparison with (\ref{compressible}), care has been taken to replace everywhere $ \nabla \times B $ as prescribed 
by (\ref{inverseconstitutive}). In (\ref{compressibleb}), the expression $ \mathscr K^{\mc}_{-1} \, B^* $ undergoes no more than one derivative. 
In view of Proposition \ref{ellllli}, this means that the corresponding contributions can be seen as acting on $ B^* $ like zero order operators 
(and even of order $-1$). 

\smallskip

\noindent The equation on $ B^* $ contains an inertial contribution at the end of the last line of (\ref{compressibleb}), which is of order $ 2 $ with 
respect to $ v $. The idea of Section \ref{vorcasincompressible} is to reduce this order to $ 1 $ by introducing certain derivatives of $ v $, namely 
those contained in the vorticity $ w $. Remarkably, the extra derivatives of $ B^* $ thus generated (when looking at the equation on $ w $) are 
exactly balanced inside symmetric structures. In the compressible case, derivatives of $ q $ (or $ \rho $) appear during this procedure. Moreover, 
all derivatives of $ v $ are required, including the divergence part $ \nabla \cdot v $. Accordingly, we have to introduce the new unknown
$$ U^{\mc}_{\mv}  := ( q, \nabla q, \nabla \cdot v, w , B^*) \in \RR \times \RR^3 \times \RR \times \RR^3 \times \RR^3 \, , \qquad w := \nabla 
\times v \, . $$
From (\ref{compressibleb}), we can deduce that
\begin{equation}\label{compressiblebcomplet}   
\left \lbrace \begin{array}{l}
\part_t q + v\cdot \nabla q + a(q) \ \nabla \cdot v = 0 \, , \vspace{5pt}\\
\part_t (\nabla q) + (v \cdot \nabla) \nabla q + a(q) \ \nabla (\nabla \cdot v) = \mathscr S^{\mc  \dot q}_{\mv 0} \, U^{\mc}_{\mv}  \, , \vspace{5pt}\\
\displaystyle \part_t (\nabla \cdot v) + (v \cdot \nabla) (\nabla \cdot v) + a(q) \ \Delta q - \nu \ \Delta (\nabla \cdot v) = \mathscr S^{\mc  \dot v}_{\mv 1} \, 
U^{\mc}_{\mv}  \, , \vspace{5pt}\\
\part_t w + v \cdot \nabla w + \bigl( \mathscr K^{\mc}_{-1} \, B^* \cdot \nabla \bigr) B^* = \mathscr S^{\mc  w}_{\mv 0} \, U^{\mc}_{\mv}  \, , \vspace{5pt}\\
\displaystyle \part_t B^* + \bigl( \bigl( v - d \, \mathscr K^{\mc}_{-1} \, B^* \bigr) \cdot \nabla \bigr) B^* + ( \mathscr K^{\mc}_{-1} \, B^* \cdot \nabla) 
w = \mathscr S^{\mc  B^*}_{\mv 0} \, U^{\mc}_{\mv}  \, .
\end{array} \right. 
\end{equation}
 In (\ref{compressiblebcomplet}), the velocity $ v $ must be deduced from $ (w , \nabla \cdot v) $ through the div-curl system (\ref{elldivcurl}),
 see Subsection \ref{div-curlsystem}. On the other hand, the operator $ \mathscr S^{\mc}_{\mv}  = {}^t (0, \mathscr S^{\mc  \dot q}_{\mv 0}, \mathscr S^{\mc  \dot v}_{\mv 0}, 
 \mathscr S^{\mc  w}_{\mv 0} , \mathscr S^{\mc  B^*}_{\mv 0} ) $ put in source term is outlined below
\begin{equation}\label{contentsource}   
\begin{array}{rl}
\displaystyle \mathscr S^{\mc  \dot q}_{\mv 0} \, U^{\mc}_{\mv}  \! \! \! & \displaystyle := \, - \nabla q \ D v - a'(q) \ (\nabla \cdot v) \ \nabla q \, , \vspace{4pt} \\
\displaystyle \mathscr S^{\mc  \dot v}_{\mv 1} \, U^{\mc}_{\mv}  \! \! \! & \displaystyle := - \sum_{i=1}^3 (\part_i v \cdot \nabla) v_i - a'(q) \ \vert \nabla q \vert^2 
- \nabla \cdot (\mathscr S^{\mc  v}_{\mv 0} \, U^{\mc}_{\mv} ) \, ,  \vspace{4pt} \\
\displaystyle \mathscr S^{\mc  w}_{\mv 0} \, U^{\mc}_{\mv}   \! \! \! & \displaystyle :=  (B^* \cdot \nabla) (\mathscr K^{\mc}_{-1} \, B^*) + (\nabla \cdot B^*_0) \ 
\mathscr K^{\mc}_{-1} \, B^* \vspace{2pt} \\
\ & \displaystyle \quad \ - \, \nabla \cdot (\mathscr K^{\mc}_{-1} \, B^*) \ B^* + (w\cdot \nabla) v - w \ (\nabla \cdot v) \, , \vspace{4pt}\\
\displaystyle \mathscr S^{\mc  B^*}_{\mv 0} U^{\mc}_{\mv}  \! \! \! & \displaystyle  := - \, \bigl( \nabla \cdot v - d \ (\nabla \cdot \mathscr K^{\mc}_{-1} \, B^*) \bigr) \ 
B^* + (\nabla \cdot B^*_0) \ \bigl( v - d \ \mathscr K^{\mc}_{-1} \, B^* \bigr)  \\
\ & \displaystyle \quad \ + \, (B^* \cdot \nabla) \bigl( v - d \ \mathscr K^{\mc}_{-1} \, B^* \bigr) - \nabla \cdot (\mathscr K^{\mc}_{-1} \, B^*) \ w + \, 
(w \cdot \nabla) (\mathscr K^{\mc}_{-1} \, B^*) \, , 
\end{array} 
\end{equation}
where $ \mathscr S^{\mc  \dot v}_{\mv 1} $ is indeed of order $ 1 $ since it is defined through $ \mathscr S^{\mc  v}_{\mv 0} $ which is given by
$$ \mathscr S^{\mc  v}_{\mv 0} \, U^{\mc}_{\mv}  := B^* \times \mathscr K^{\mc}_{-1} \, B^* + \frac{1}{2} \ \nabla \vert \mathscr K^{\mc}_{-1} \, B^* \vert^2 \, . $$
In comparison with $ \mathscr S^\mi_{\mv 0} \, U^\mi_{\mv} $, we have many added new contributions. Some of them come from the fact that $ \mathscr K^{\mc}_{-1} \, B^* $
and $ B^* $ are no more solenoidal vector fields. Observe that we have exploited Lemma \ref{divfreepreser} to replace $ \nabla \cdot B^* $ everywhere by 
$ \nabla \cdot B^*_0 $ (this is essential to avoid artificial losses of derivatives coming from $ \nabla \cdot B^* $). On the other hand, we have incorporated 
the vortex stretching induced by the term $ \nabla \cdot v $ (which is no more zero) as well as other contributions related to $ \nabla \cdot v $. 
  

\subsection{Proof of Theorem \ref{theoprin}} \label{lwellposedcomp} We start by showing Theorem \ref{theoprin} under the more restrictive 
regularity assumption $ s > 7/2 $. We refer to Section \ref{potentialformulations} for the optimal result. From (\ref{reginidata}), we know that
$ \rho_0 \in \cH^s $, and by construction we have $ q_0 = g (\bar \rho + \rho_0) $ with $ g(\bar \rho) = 0 $. Then, by the rule of composition in 
$ \cH^s $, we recover that $ q_0 \in \cH^s $. In particular, at time $ t = 0 $, with $ \tilde s := s-1 > 5/2 $, we can assert that
\begin{equation}\label{inidatasimplbis}  
U^{\mc}_{\mv}  (0,\cdot) = U^{\mc}_{\mv  0} = (q_0, \nabla q_0,\nabla \cdot v_0, w_0 , B^*_0) \in \cH^{\tilde s} \, , \qquad w_0 := \nabla \times v_0  \, .
\end{equation}
Any smooth solution to (\ref{compressible})-(\ref{lienBB*ini})-(\ref{inidata}) leads to a solution to (\ref{compressiblebcomplet})-(\ref{inidatasimplbis}), 
and conversely. We study below the time evolution of the $ L^2 $-norm of $ U^{\mc}_{\mv}  $ (assuming for the moment that $ U^{\mc}_{\mv}  $ is bounded in the large 
$ \cH^{\tilde s} $-norm).

\begin{lem} \label{L2inestimatebis} [$L^2 $-energy estimate for the vorticity formulation in the compressible case] Let $ T > 0 $.  Assume 
that the function $ U^{\mc}_{\mv}  \in C([0,T];\cH^{\tilde s}) $ is a solution to (\ref{compressiblebcomplet}) with initial data as in (\ref{inidatasimplbis}). 
Then, we can find a constant $ C $ depending only on the $ C([0,T];\cH^{\tilde s}) $-norm of $ U^{\mc}_{\mv}  $ such that
\begin{equation}\label{0-energy estimatebis}  
\parallel U^{\mc}_{\mv}  (t,\cdot) \parallel_{L^2} \leq \parallel U^{\mc}_{\mv  0} \parallel_{L^2} \, e^{C \, t + C \, t / \nu} \, , \quad \forall \, t \in [0,T]  \, .
\end{equation}
\end{lem} 

\begin{proof} Multiply (\ref{compressiblebcomplet}) by $ {}^t U^{\mc}_{\mv}  $, and then integrate with respect to $ x $. After integrations by parts, we find
\begin{equation}\label{bilanener}   
\begin{array}{l} 
\displaystyle \frac{1}{2} \ \frac{d}{dt} \bigg( \int_{\RR^3} \vert U^{\mc}_{\mv}  (t,\cdot)  \vert^2 \ dx \bigg) + \, \nu \int_{\RR^3} \vert \nabla (\nabla \cdot v) \vert^2 \ dx \\
\displaystyle \qquad \quad = \int_{\RR^3} \bigl( \nabla q \cdot \mathscr S^{\mc  \dot q}_{\mv 0} \, U^{\mc}_{\mv} + (\nabla \cdot v) \ \mathscr S^{\mc  \dot v}_{\mv 1} \, 
U^{\mc}_{\mv} + w \cdot \mathscr S^{\mc  w}_{\mv 0} \, U^{\mc}_{\mv} + B^* \cdot \mathscr S^{\mc  B^*}_{\mv 0} \, U^{\mc}_{\mv} \bigr) \ dx \\
\displaystyle \qquad \qquad + \int_{\RR^3} ( h_0 + h_1 + h_2 + h_3 + h_4 ) \ dx \, , 
\end{array} 
\end{equation}
where the $ \mathscr S^{\mc  \star}_{\mv  \diamond} $ are the source terms of (\ref{contentsource}), whereas the $ h_\star $ come from the quasilinear 
parts. We find that
$$ \begin{array}{l} 
\displaystyle \int_{\RR^3} h_0 \ dx := - \int q \ a(q) \ (\nabla \cdot v) \ dx \, , \medskip \\
\displaystyle \int_{\RR^3} h_1 \ dx := - \int U^{\mc}_{\mv}  \cdot \bigl( (v\cdot \nabla) U^{\mc}_{\mv}  \bigr) \ dx = \frac{1}{2} \int ( \nabla \cdot v) \ \vert 
U^{\mc}_{\mv}  \vert^2 \ dx \, , \medskip \\
\displaystyle \int_{\RR^3} h_2 \ dx :=  - \int_{\RR^3} a(q) \ \bigl( \nabla q \cdot \nabla (\nabla \cdot v) + (\nabla \cdot v) \ \Delta q 
\bigr) \ dx = \int_{\RR^3} a'(q) \ (\nabla \cdot v) \ \vert \nabla q \vert^2 \ dx \, ,  \medskip \\
\displaystyle \int_{\RR^3} h_3 \ dx := \int_{\RR^3}  \bigl( \nabla \cdot (\mathscr K^{\mc}_{-1} \, B^*) \bigr) \ (w \cdot B^*) 
\ dx \, , \medskip \\
\displaystyle \int_{\RR^3} h_4 \ dx := - \frac{d}{2} \int_{\RR^3}  \bigl( \nabla \cdot (\mathscr K^{\mc}_{-1} \, B^*) \bigr) \ 
\vert B^* \vert^2 \ dx \, .
\end{array} $$
We consider each term separately. Knowing that $ \tilde s > 5/2 $, we use repeatedly the Sobolev embedding theorem $ \cH^{\tilde s} 
\hookrightarrow L^\infty $. We also exploit Proposition \ref{ellllli} with $ r = s-2 $ to deduce from $ \nabla \times B^* \in \cH^{\tilde s -1} 
\equiv \cH^{s-2} $ that $ \mathscr K^{\mc}_{-1} \, B^* = (\mathscr  L^{\mc}_2)^{-1} \, \nabla \times B^* \in \cH^s $. In other words
\begin{equation}\label{condrerr}    
\qquad \parallel \mathscr K^{\mc}_{-1} \, B^* \parallel_{\cH^{\tilde s}} \lesssim \parallel B^* \parallel_{\cH^{\tilde s-1}} \, , \qquad 
\parallel \nabla \mathscr K^{\mc}_{-1} \, B^* \parallel_{\cH^{\tilde s}} \lesssim \parallel B^* \parallel_{\cH^{\tilde s}} . 
\end{equation}
On the other hand, by assumption, we know that 
$$ \parallel U^{\mc}_{\mv}  \parallel_{L^\infty ([0,T] \times \RR^3)} + \parallel \nabla U^{\mc}_{\mv}  \parallel_{L^\infty ([0,T] \times \RR^3)} 
\lesssim M:= \parallel U^{\mc}_{\mv}  \parallel_{C([0,T];\cH^{\tilde s})} < + \infty \, . $$
First, from Lemma \ref{elldpourdivw}, we have 
$$ \bigg\vert \int_{\RR^3} \nabla q \cdot \mathscr S^{\mc  \dot q}_{\mv 0} \, U^{\mc}_{\mv} \ dx \bigg\vert \leq \, \bigl( \parallel \nabla q \parallel_{L^\infty} + 
\parallel \nabla a(q) \parallel_{L^\infty} \bigr) \, \parallel U^{\mc}_{\mv}  \parallel^2_{L^2} \, \lesssim \,  M \, \parallel U^{\mc}_{\mv}  \parallel^2_{L^2} . $$
We now turn to the equation on $ \nabla \cdot v $, where the source term $ \mathscr S^{\mc  \dot v}_{\mv 1} U^{\mc}_{\mv}  $ does include a loss
of one derivative. The idea is to compensate this after integration by parts by the bulk (fluid) viscosity. 
With the help of Lemma \ref{elldpourdivw}, this gives rise to 
$$ \begin{array}{rl}
\displaystyle \bigg\vert \int_{\RR^3} (\nabla \cdot v) \ \mathscr S^{\mc  \dot v}_{\mv 1} U^{\mc}_{\mv} \ dx \bigg \vert \leq \! \! \! & \displaystyle \sum_{i=1}^3 
\int_{\RR^3} \vert \nabla \cdot v \vert \ \vert \part_i v \vert \ \vert \nabla v_i \vert \ dx + \int_{\RR^3} \vert \nabla \cdot v \vert \ \vert a'(q) 
\vert \ \vert \nabla q \vert^2 \ dx \\
\ & \displaystyle + \int_{\RR^3} \vert \nabla (\nabla \cdot v) \vert \ \vert \mathscr S^{\mc  v}_{\mv 0} U^{\mc}_{\mv}  \vert \ dx \smallskip \\
\ \leq \! \! \!  & \displaystyle C \ \bigl( \parallel U^{\mc}_{\mv}  \parallel_{L^\infty} \bigr) \ \parallel U^{\mc}_{\mv}  \parallel^2_{L^2} \\
\ & \displaystyle + \, c \ \nu  \int_{\RR^3} \vert \nabla (\nabla \cdot v) \vert^2 \ dx + \frac{C}{\nu} \int_{\RR^3} \vert \mathscr S^{\mc  v}_{\mv 0} 
U^{\mc}_{\mv}  \vert^2 \ dx \, ,
\end{array} $$
where the constant $ c \in \RR_+^* $ can be chosen as small as wished.  Observe that
$$ \int_{\RR^3} \vert \mathscr S^{\mc  v}_{\mv 0} U^{\mc}_{\mv}  \vert^2 \ dx \leq \parallel \mathscr S^{\mc  v}_{\mv 0} U^{\mc}_{\mv}  \parallel_{L^\infty} \int_{\RR^3}
 \vert \mathscr S^{\mc  v}_{\mv 0} U^{\mc}_{\mv}  \vert \ dx \, , $$
From (\ref{condrerr}), we have
$$ \begin{array}{rl} 
\displaystyle \parallel \mathscr S^{\mc  v}_{\mv 0} U^{\mc}_{\mv}  \parallel_{L^\infty} \! \! \! & \displaystyle \lesssim \parallel \mathscr K^{\mc}_{-1} \, B^* 
\parallel_{L^\infty} \, \bigl( \parallel B^* \parallel_{L^\infty} + \parallel \nabla \mathscr K^{\mc}_{-1} \, B^* \parallel_{L^\infty}  \bigr) \\
\ & \displaystyle \lesssim \parallel \mathscr K^{\mc}_{-1} \, B^* \parallel_{\cH^{\tilde s}} \, \bigl( \parallel B^* \parallel_{\cH^{\tilde s}} + \parallel 
\nabla \mathscr K^{\mc}_{-1} \, B^* \parallel_{\cH^{\tilde s}}  \bigr) \lesssim \parallel B^* \parallel^2_{\cH^{\tilde s}} \lesssim M^2 \, .
\end{array} $$
On the other hand
$$ \int_{\RR^3} \vert \mathscr S^{\mc  v}_{\mv 0} U^{\mc}_{\mv}  \vert \ dx \leq \parallel \mathscr K^{\mc}_{-1} \, B^* \parallel_{L^2} \, \bigl( \parallel B^* 
\parallel_{L^2} + \parallel \nabla \mathscr K^{\mc}_{-1} \, B^* \parallel_{L^2}  \bigr) \lesssim \parallel U^{\mc}_{\mv}  \parallel^2_{L^2} . $$
Thus, we can retain that
$$  \bigg\vert \int_{\RR^3} (\nabla \cdot v) \ \mathscr S^{\mc  \dot v}_{\mv 1} \, U^{\mc}_{\mv} \ dx \bigg \vert \leq C \ M \ \Bigl( 1 + \frac{M}{\nu} \Bigr) \parallel U^{\mc}_{\mv}  
\parallel^2_{L^2} + \, c \ \nu  \int_{\RR^3} \vert \nabla (\nabla \cdot v) \vert^2 \ dx \, . $$
Thanks to Lemma \ref{preservationB*} and (\ref{condrerr}), the two contributions $ \mathscr S^{\mc  w}_{\mv 0} \, U^{\mc}_{\mv} $ and 
$ \mathscr S^{\mc  B^*}_{\mv 0} \, U^{\mc}_{\mv} $ are both of order $ 0 $ in terms of $ U^{\mc}_{\mv} $. Be careful, in the two definitions of 
$ \mathscr S^{\mc  w}_{\mv 0} \, U^{\mc}_{\mv} $ and $ \mathscr S^{\mc  B^*}_{\mv 0} \, U^{\mc}_{\mv} $ some derivatives act on 
$ \mathscr K^{\mc}_{-1} $ and therefore on $ \rho $ (since the operator $ \mathscr K^{\mc}_{-1} $ involves coefficients depending on $ \rho $). But these one 
order derivatives of $ \rho $ are included in $ U^{\mc}_{\mv}  $ (through the derivatives of $ q $), and therefore this corresponds indeed
to zero order contributions.

\smallskip

\noindent As described above, we can obtain
$$ \bigg\vert \int_{\RR^3} w \cdot  \mathscr S^{\mc  w}_{\mv 0} \, U^{\mc}_{\mv} \ dx \bigg\vert \lesssim \,  M \, \parallel U^{\mc}_{\mv}  \parallel^2_{L^2} \, , \qquad 
\bigg\vert \int_{\RR^3} B^* \cdot \mathscr S^{\mc  B^*}_{\mv 0} \, U^{\mc}_{\mv} \ dx \bigg\vert \lesssim \,  M \, \parallel U^{\mc}_{\mv}  \parallel^2_{L^2} . $$
The same sort of arguments applies to handle the $ h_\star $. We find that
$$ \bigg\vert \int_{\RR^3} h_j \ dx \bigg\vert \lesssim \,  M \, \parallel U^{\mc}_{\mv}  \parallel^2_{L^2} \, , \qquad \forall \, j \in \{0, \cdots,4 \} \, .
$$
By selecting $ c $ small enough, we can absorb the term implying the $ L^2 $-norm of $ \nabla (\nabla \cdot v) $. Note that the presence of 
a bulk (fluid) viscosity is crucial here to compensate for losses related to $ \nabla (\nabla \cdot v) $. At the end, there remains 
$$ \frac{d}{dt} \bigg( \int_{\RR^3} \vert U^{\mc}_{\mv}  (t,\cdot)  \vert^2 \ dx \bigg) \leq C \ M \ \Bigl( 1 + \frac{M}{\nu} \Bigr) \ \int_{\RR^3} 
\vert U^{\mc}_{\mv}  (t,\cdot)  \vert^2 \ dx \, . $$
It suffices to implement Gr\"onwall's inequality to recover the inequality (\ref{0-energy estimatebis}) for some convenient constant $ C $.
\end{proof}

\noindent Starting from there, the construction of $ \cH^{\tilde s} $-solutions to (\ref{compressiblebcomplet}) can be achieved through the 
general strategy \cite{Alin,Maj84} already explained at the end of Section \ref{vorcasincompressible}. Note that the lifespan $ T $ does 
depend on the bulk viscosity $ \nu $. It shrinks to $ 0 $ when $ \nu $ goes to $ 0 + $. Then, from the $ \cH^{\tilde s} $-solution to 
(\ref{compressiblebcomplet}), we can recover solutions to (\ref{compressible}) which are such that $ (q,v,B^*) \in \cH^s \times  \cH^s \times 
\cH^{s-1} $ with $ s > 7/2 $. Starting from there and from (\ref{inverseconstitutive}), we get that $ \nabla \times B \in \cH^s $, while 
$ \nabla \cdot B= \nabla \cdot B_0^* \in \cH^s $. It follows that $ B \in \cH^{s+1} $ as indicated. This concludes the proof of Theorem 
\ref{theoprin} at least when $ s > 7/2 $. The case $ s > 5/2 $ is investigated in the next section.


\section{The potential formulations} \label{potentialformulations} 
The aim of this section is to develop an alternative to the vorticity formulations. The basic idea is to integrate $ B^* $ instead of looking 
at derivatives of $ (q,v) $. From $ \mathscr P B^* $, or just from $ B^* $ when $ \nabla \cdot B^* = 0 $, 
we can extract a {\it magnetic potential} $ A^* $ which is defined by
\begin{equation}\label{definedbyA*}  
\nabla \times A^* = \mathscr P B^* \, , \qquad \nabla \cdot A^* = 0 \, . 
\end{equation}
From now on, we imply $ \rho $, $ v $ and $ A^* $ as new unknowns, with the following motivations:
\begin{itemize}
\item [-] {\it A simplified presentation.} Like in (\ref{compressible}), the (compressible) potential formulation involves $ 7 $ unknowns, 
namely the components of $ \rho $, $ v $ and $ A^* $. Contrary to (\ref{compressiblebcomplet}), there is no need of introducing four 
supplementary unknowns.
\item [-] {\it A better regularity result.} Up to now, we have worked with $ s > 7/2 $. As stated in Theorems \ref{wellposednessideal} and 
\ref{theoprin}, we would like to improve the threshold up to $ s > 5/2 $. To this end, we will avoid the use of Proposition \ref{ellllli} which 
is costly in terms of the regularity of $ \rho $.
\item [-] {\it Additional insights.} This change of point of view offers complementary perspectives. For instance, in the compressible case,
it allows to better understand why a bulk (fluid) viscosity is required for stability. 
\end{itemize}

\noindent The potential formulations are therefore more simple in appearance and more efficient in some respects. However, there 
are important subtleties when performing $ L^2 $-energy estimates, explaining why this approach has been postponed until now. 
On the other hand, the vorticity formulations are necessary and instructive to understand how higher order energy estimates can 
be performed at the level of the potential formulations.

\smallskip

\noindent From (\ref{retainthat}) and (\ref{definedbyA*}), we deduce that $ \part_t B^* = \nabla \times \part_t A^* $. It follows that the curl 
operator can be put in factor of the last equation of (\ref{compressible}), where it can be cancelled (modulo a gradient). This idea applies 
quite directly in the incompressible context of Subsection \ref{potcasincompressible}. It must be carefully implemented in the compressible framework 
of Subsection \ref{potcompressibleframework}.


\subsection{The incompressible situation} \label{potcasincompressible} The condition $ \nabla \cdot B^* = 0 $ and the constant density make things 
easier. In Paragraph \ref{paraideal0}, we derive the incompressible potential equations. In Paragraph \ref{paraideal1}, we perform $ L^2 $-energy 
estimates on linearized equations. In Paragraph \ref{paraideal2}, we conclude the proof of Theorem \ref{wellposednessideal}.


\subsubsection{The incompressible potential equations} \label{paraideal0} The unknown is $ U_{\mathfrak p}^\mi := (v,A^*) $. We consider the system
\begin{equation}\label{syssimpliidealpot}  
\left \lbrace \begin{array}{l}
  \displaystyle \part_t v + (v\cdot \nabla) v - (A^*-A) \times (\nabla \times A^*) + \nabla p = 0 \, , 
  \medskip \\
\displaystyle \part_t A^* - \bigl(v - d \, (A^*-A) \bigr) \times ( \nabla \times A^*) - (A^*-A) \times (\nabla \times v) + \nabla e = 0 \, ,
\end{array} \right. 
\end{equation}
where both $ v $ and $ A^* $ are solenoidal vector fields
\begin{equation}\label{solenoidalvectorfields}  
\nabla \cdot v = 0 \, , \qquad \nabla \cdot A^* = 0 \, .
\end{equation}
while $ A $ can be obtained from $ A^* $ through 
\begin{equation}\label{lienBB*2idealpot}  
A = (\id-\Delta)^{-1} A^* \, , \qquad \nabla \cdot A = 0 \, ,
\end{equation}
The introduction of the Lagrange multipliers (scalar functions) $ p $ and $ e $ is needed above to ensure the propagation of the constraints $ \nabla \cdot v = 0 $
and $ \nabla \cdot A^* = 0 $.

\begin{lem} \label{potpotideal} [Link between the incompressible potential and vorticity formulations] Let $ U_{\mathfrak p}^\mi = (v,A^* ) $ be some 
$ \cH^s $-solution on $ [0,T] $ to (\ref{syssimpliidealpot})-(\ref{lienBB*2idealpot})-(\ref{solenoidalvectorfields}) with $ s > 5/2 $. Define
\begin{equation}\label{defineBB*}  
B^* := \nabla \times A^* \, , \qquad B :=  \nabla \times A \, .
\end{equation}
Then, $ (v,B^*,B) $ is a solution on $ [0,T] $ to (\ref{divfreeini})-(\ref{syssimpli})-(\ref{lienBB*2}), which is as in (\ref{reginidatainpropa}).
  \end{lem} 

\begin{proof} By construction, we have (\ref{divfreeini}). On the other hand, by applying the curl operator to (\ref{lienBB*2idealpot}), 
we get (\ref{lienBB*2}). Since $ \nabla \cdot A = 0 $, the relation (\ref{lienBB*2idealpot}) is the same as 
$$ A^* - A = - \Delta A = \nabla \times (\nabla \times A) = \nabla \times B \, . $$
It follows that
$$ - (A^*-A) \times (\nabla \times A^*) =  B^* \times (\nabla \times B) \, . $$
Exploiting the two above relations and applying the curl operator $ \nabla \times $ to the second line of (\ref{syssimpliidealpot}),
the term $ \nabla e $ disappears and we have directly access to (\ref{syssimpli}).
\end{proof} 

\noindent At the initial time $ t = 0 $, we impose
\begin{equation}\label{inidatasimplpotnodot}  
U^\mi_{\mathfrak p} (0,\cdot) = U^\mi_{\mathfrak p 0} = (v_0, A^*_0) \in \cD^{s} (\RR^3;\RR^3)^2 \, , \qquad s > 5/2  \, .
\end{equation}


\subsubsection{$ L^2 $-energy estimates} \label{paraideal1} The conserved quantity $ \cE^\mi $ of Lemma \ref{energypreserin},
see the definition (\ref{conservedprop}), can be reformulated according to 
\begin{equation}\label{finiteenerg} 
\cE^\mi := \frac 1 2\int_{\RR^3} \big(\vert v \vert^2 + \vert \nabla \times (\id-\Delta)^{-1} A^* \vert^2 + \vert \Delta \, (\id-\Delta)^{-1} 
A^* \vert^2 \big) \ dx < + \infty \, . 
\end{equation}
This already furnishes some (high frequency) $ L^2 $-bound concerning $ U_{\mathfrak p}^\mi $. But this is 
not enough. To construct solutions by a fixed point argument, we also need to consider the stability issue. To this end, we have to 
look at the linearized equations coming from (\ref{syssimpliidealpot}), dealing with $ \dot U^\mi_{\mathfrak p} = ( \dot v , \dot A^*) $. 
When doing this, the term which for instance is at top right of (\ref{syssimpliidealpot}) leads to
$$ - (A^*-A) \times (\nabla \times \dot A^*) + (\nabla \times A^*) \times \bigl( \dot A^* - (\id - \Delta)^{-1} \dot A^*\bigr) \, . $$
Given a Lipschitz field $ A^* $, the right hand side is such that
$$ \parallel (\nabla \times A^*) \times \bigl( \dot A^* - (\id - \Delta)^{-1} \dot A^*\bigr) \parallel_{L^2} \lesssim  \parallel \dot U^\mi_{\mathfrak p}  
\parallel_{L^2} \, . $$
Such contributions clearly cannot undermine the local $ L^2 $-stability. Thus, to simplify the presentation, they can be ignored. We 
can focus on
\begin{equation}\label{syssimpliidealpotdot}  
\left \lbrace \begin{array}{l}
  \displaystyle \part_t \dot v + (v\cdot \nabla) \dot v + \nabla \dot p - (A^*-A) \times (\nabla \times \dot A^*) = 0 \, , 
  \medskip \\
\displaystyle \part_t \dot A^* - \bigl(v - d \, (A^*-A) \bigr) \times ( \nabla \times \dot A^*) - (A^*-A) \times (\nabla \times \dot v) + \nabla \dot e= 0 \, ,
\end{array} \right. 
\end{equation}
together with
\begin{equation}\label{solenoidalvectorfieldsdot}  
\nabla \cdot \dot v = 0 \, , \qquad \nabla \cdot \dot A^* = 0 \, .
\end{equation}
At the initial time $ t = 0 $, we impose
\begin{equation}\label{inidatasimplpotdot}  
\dot U^\mi_{\mathfrak p} (0,\cdot) = \dot U^\mi_{\mathfrak p 0} = (\dot v_0, \dot A^*_0) \in \cD^0 (\RR^3;\RR^3)^2 \, .
\end{equation}

\begin{lem} \label{L2inestimatepotre} [$L^2 $-energy estimates for the linearized incompressible  potential equations] Let $ T > 0 $.  Assume 
that $ U^\mi_{\mathfrak p} = (v,A^*) $ is such that $ U^\mi_{\mathfrak p} \in C([0,T];\cD^s) $ for some $ s> 5/2 $. Then, the Cauchy problem 
(\ref{syssimpliidealpotdot})-(\ref{solenoidalvectorfieldsdot}) with initial data (\ref{inidatasimplpotdot}) has a solution on $ [0,T] $. 
Moreover, we can find a constant $ C $ depending only on the 
$ C([0,T];\cH^s) $-norm of $ U^\mi_{\mathfrak p} $ such that
\begin{equation}\label{0-energy estimatepot}  
\parallel \dot U^\mi_{\mathfrak p} (t,\cdot) \parallel_{L^2} \leq \parallel \dot U^\mi_{\mathfrak p 0} \parallel_{L^2} \, e^{C \, t} \, , \qquad \forall \, t \in [0,T]  \, .
\end{equation}
\end{lem} 

\noindent Any $ \cH^s $-solution to the initial value problem 
(\ref{syssimpliidealpot})-(\ref{lienBB*2idealpot})-(\ref{solenoidalvectorfields})-(\ref{inidatasimplpotnodot})
leads to a solution to (\ref{syssimpliidealpotdot})-(\ref{solenoidalvectorfieldsdot}) with initial data $ \dot U^\mi_{\mathfrak p 0} = U^\mi_{\mathfrak p 0} $. As a 
consequence, the proof of Lemma \ref{L2inestimatepotre} gives another access to some $ L^2 $-bound, namely
$$ \parallel U^\mi_{\mathfrak p} (t,\cdot) \parallel_{L^2} \leq \parallel U^\mi_{\mathfrak p 0} \parallel_{L^2} \, e^{C \, t} \, , \qquad \forall \, t \in [0,T]  \, . $$

\begin{proof} To gain a better grasp of the arguments, we have made a clear distinction between two kinds of quantities:
\begin{itemize}
\item [-] On the one hand, there are those which play the role of coefficients and which are managed through the assumption
\begin{equation}\label{ass-coef}   
\parallel U^\mi_{\mathfrak p} \parallel_{L^\infty ([0,T] \times \RR^3)} + \parallel \nabla U^\mi_{\mathfrak p} \parallel_{L^\infty ([0,T] \times \RR^3)} 
\lesssim M:= \parallel U^\mi_{\mathfrak p} \parallel_{C([0,T];\cH^s)} < + \infty \, .
\end{equation}
\item [-] On the other hand, there are those which are handled as unknowns and which are market by a dot, like $ \dot U^\mi_{\mathfrak p} = (\dot v,\dot A^*) $. 
\end{itemize}

\noindent We now exploit the formalism of Remark \ref{yeti}.  We denote by $ \mathscr T_C $ with $ C = A^*-A $ or $ C = v $ the operator 
defined at the level of (\ref{decompoooo}). We find that $ \mathscr T_C^* \dot A^* = - C \times (\nabla \times \dot A^*) $. Thus, the system 
(\ref{syssimpliidealpotdot}) can be rewritten according to
\begin{equation}\label{syssimpliidealrewritten}  
\left \lbrace \begin{array}{l}
  \displaystyle \part_t \dot v + (v\cdot \nabla) \dot v + \nabla \dot p + \mathscr T_{A^*-A}^* \dot A^* = 0 \, , 
  \medskip \\
\displaystyle \part_t \dot A^* + \mathscr T_v^* \dot A^*- d \ \mathscr T_{A^*-A}^* \dot A^* + \mathscr T_{A^*-A}^* \dot v + \nabla \dot e = 0 \, .
\end{array} \right. 
\end{equation}
As already noted, the operator $ \mathscr T_C $ is not skew-adjoint and, of course, neither is $ \mathscr T_C^* $. But (Remark \ref{yeti}), 
knowing that the contribution $ \nabla \cdot B^* $ is given (or can be forgotten), the action of $ \mathscr T_C $ (viewed as $ \tilde {\mathscr T}_C $) 
on $ B^* $ becomes skew-adjoint. 
This argument was crucial in Sections \ref{vorcasincompressible} and \ref{vorcompressibleframework}. None of that applies to the action of 
$ \mathscr T_C^* $ on $ \dot A^* $ (because the analogue of $ \nabla \cdot B^* = 0 $ in the context of $ \mathscr T_C^* \dot A^* $ is not 
$ \nabla \cdot \dot A^* = 0 $). In other words, (\ref{syssimpliidealrewritten}) is not well-posed, while its dual version is, in the sense that it 
becomes symmetric under the condition (\ref{solenoidalvectorfieldsdot}). To put this principle into practice, we multiply the first and second 
equation of (\ref{syssimpliidealrewritten}) respectively by $ \dot v $ and $ \dot A^* $; we integrate with respect to the variable $ x $; and then 
we force the emergence of the operator $ \mathscr T_C $ (instead of $ \mathscr T_C^* $) by passing to the adjoint. Since $ v \in \cD^{s} $, 
the contribution related to $ v \cdot \nabla $ disappears. Since $ \nabla \cdot \dot A^* = 0 $, the term involving $  \nabla \dot e $ is eliminated. 
Denoting by $ \langle \cdot, \cdot \rangle $ the scalar product in $ L^2 $, this furnishes
 $$   \frac{1}{2} \ \frac{d}{dt}  \bigg( \int_{\RR^3} \vert \dot U^\mi_{\mathfrak p} (t,\cdot)  \vert^2 \ dx  \bigg) = -  \, \langle  \mathscr T_{A^*-A} \dot v , \dot A^*\rangle 
 - \langle  \mathscr T_v \dot A^*, \dot A^*\rangle +\langle  \mathscr T_{A^*-A} \dot A^* , d \, \dot A^* - \dot v \rangle \, . $$
 Observe the changeover from $ \mathscr T_C^* $ to $ \mathscr T_C $. From there, the decomposition (\ref{definitionufbis}) becomes pertinent. 
 In the actual incompressible situation, using (\ref{followingidentity}), this yields
  $$ \begin{array}{rl} 
\displaystyle \frac{1}{2} \ \frac{d}{dt}  \bigg( \int_{\RR^3} \vert \dot U^\mi_{\mathfrak p} (t,\cdot)  \vert^2 \ dx  \bigg) = \! \! \! & \displaystyle \frac{1}{2} 
\int_{\RR^3} \Bigl( \bigl( (A^* - A) \cdot \nabla \bigr) \bigl( \dot A^* \cdot ( 2 \, \dot v - d \, \dot A^* ) \bigr) + (v\cdot \nabla ) \vert \dot A^* \vert^2 
\Bigr) \ dx \smallskip \\
\ & \displaystyle - \langle (\dot v \cdot \nabla) (A^*-A) , \dot A^* \rangle  - \langle (\dot A^* \cdot \nabla) v , \dot A^* \rangle  \smallskip \\
\ & \displaystyle + \langle (\dot A^* \cdot \nabla) (A^*-A) , d \, \dot A^* - \dot v \rangle \, . 
\end{array} $$
In the right hand side, since both $ A^* -A $ and $ v $ are solenoidal vector fields, after integration by parts, the first line just disappears. Exploiting 
(\ref{ass-coef}) to control the two last lines, we find that
 $$ \frac{d}{dt}  \bigg( \int_{\RR^3} \vert \dot U^\mi_{\mathfrak p} (t,\cdot)  \vert^2 \ dx  \bigg) \lesssim  \int_{\RR^3} \vert \dot U^\mi_{\mathfrak p} (t,\cdot)  \vert^2 \ dx \, . $$
 By Gr\"onwall's inequality, we recover (\ref{0-energy estimatepot}).
\end{proof}


\subsubsection{End of the proof of Theorem \ref{wellposednessideal}} \label{paraideal2} The construction of $ \cH^s $-solutions to (\ref{syssimpliidealpot}) 
follows the (general standard) lines mentioned before. It will not be detailed. But, we would like to add just a few words about how higher order estimates
can be obtained. In the context of (\ref{syssimpliidealpot}), there are two ways to proceed:
\begin{itemize}
\item [-] The incompressible vorticity formulation is in fact similar to a derived version of the incompressible potential formulation. It follows that the preceding $ L^2 $-estimates 
for (\ref{syssimplinet})-(\ref{divfreekeep}) correspond to one order estimates for (\ref{syssimpliidealpot})-(\ref{lienBB*2idealpot}). In other words, the work of 
 Subsection \ref{lwellposedin} can be seen as the first stage to check that higher order estimates (namely $ \cH^1 $-estimates) are available for the incompressible 
 potential formulation. 
\item [-] Looking at the linearized equations (\ref{syssimpliidealrewritten}) with adequate source terms amounts to the same thing as studying the equations 
satisfied by the derivatives of $ U^\mi_{\mathfrak p} $. Thus, the proof of Lemma \ref{L2inestimate} already provides with convincing points of reference towards 
$ \cH^s $-estimates.
\end{itemize}


\subsection{The compressible framework} \label{potcompressibleframework} The general lines are as in Subsection \ref{potcasincompressible} but the 
variations of $ \rho $ oblige to adapt a number of aspects. The first step (in Paragraph \ref{paracomp1}) is to correctly interpret (\ref{lienBB*ini}) in terms of the 
potentials $ A^* $ and $ A $. The second stage (in Paragraph \ref{paracomp2}) is to propose potential equations that are compatible with (\ref{compressible}).
Then (in Paragraph \ref{paracomp3}), we explain how to obtain $ L^2 $-energy estimates on linearized equations.  


\subsubsection{The potential constitutive relation} \label{paracomp1} Keeping (\ref{definedbyA*}) and assuming that $ \mathscr P B 
= \nabla \times A $, the constitutive relation (\ref{lienBB*ini}) is the 
same as
$$ \mathscr P B^* - \mathscr P B - \nabla \times \bigg( \frac{\nabla \times B}{\rho(x)} \bigg) = 
\nabla \times \bigg( A^* - A - \frac{\nabla \times B}{\rho(x)} \bigg) = 0 \, . $$
This suggests to impose 
\begin{equation}\label{lienAA*inireprem} 
A^* - A - \frac{\nabla \times B}{\rho(x)} = 0 \, . 
\end{equation}
Exploiting (\ref{lienBB*inire}) with again (\ref{defineBB*}), this is equivalent to
\begin{equation}\label{lienAA*inire}  
 A = A^* - ( \mathscr L^{\mc}_2)^{-1} \, \nabla \times (\nabla \times A^*) = ( \mathscr L^{\mc}_2)^{-1} ( \rho \, A^*) \, .
\end{equation}
By this way, $ A $ is deduced from $ A^* $ through a pseudo-differential operator which is of order zero (or less). In what follows, 
we do not need Proposition \ref{ellllli}. Instead, we are satisfied with Lemma \ref{zeroorder} leading to
\begin{equation}\label{continuouspassaa*}  
\parallel A \parallel_{L^2} \lesssim \parallel A^* \parallel_{L^2} \, . 
\end{equation}
One of the difficulties is to show that the choice (\ref{lienAA*inire}) is appropriate.


\subsubsection{The compressible potential equations} \label{paracomp2} Let $ p : \RR \rightarrow \RR $ be any smooth 
function (of $ \rho $). Fix some initial data $ (\rho_0,v_0,B_0^*) $ as in (\ref{reginidata}). From $ B_0^* $, extract the vector 
field $ A_0^* $ which is such that
$$ \nabla \times A_0^* = \mathscr P \, B_0^* \, , \qquad \nabla \cdot A_0^* = 0 \, , \qquad A_0^* \in \cH^s \, . $$
The unknown is $ U_{\mathfrak p}^{\mc} := (\rho,v,A^*) $. Consider the system
\begin{equation}\label{syssimplicompot}  
\left \lbrace \begin{array}{l}
\part_t \rho + (v\cdot \nabla) \rho + \rho \, \nabla \cdot v = 0 \, , \medskip \\
  \displaystyle \part_t v + (v\cdot \nabla ) v + \frac{\nabla p}{\rho} - (A^*-A) \times (\nabla \times A^*) \\
  \qquad \qquad \qquad \qquad \quad + \, \nabla \bigl( \vert A^*-  A \vert^2 /2\bigr) = \nu \ \nabla (\nabla \cdot v) 
  + (A^*-A) \times \mathscr Q B_0^* \, ,  \medskip \\
\displaystyle \part_t A^* - \bigl(v - d \, (A^*-A) \bigr) \times ( \nabla \times A^*) - (A^*-A) \times (\nabla \times v) + \nabla e  \\
\qquad \qquad \qquad \qquad \qquad \qquad \qquad \qquad = \bigl(v - d \, (A^*-A) \bigr) \times \mathscr Q B_0^* \, ,
\end{array} \right. 
\end{equation}
together with 
\begin{equation}\label{ehhhhoui}  
 \nabla \cdot A^* = 0 \, , 
 \end{equation}
where $ A $ is deduced from $ A^* $ through (\ref{lienAA*inire}). At the initial time $ t = 0 $, we impose
\begin{equation}\label{inidatasimplpotnodotcomp}  
U^{\mc}_{\mathfrak p} (0,\cdot) = U^{\mc}_{\mathfrak p 0} = (\rho_0,v_0, A^*_0) \in \cH^s(\RR^3;\RR) \times \cH^s (\RR^3;\RR^3)^2 \, , \qquad s > 5/2  \, .
\end{equation}

\begin{lem} \label{Link between the compressible potential} [Link between the compressible potential and vorticity formulations] 
Let $ U_{\mathfrak p}^{\mc} $ be some $ \cH^s $-solution on $ [0,T] $ to (\ref{lienAA*inire})-(\ref{syssimplicompot}) with initial data as in 
(\ref{inidatasimplpotnodotcomp}). Define
\begin{equation}\label{defineBB*autre}  
B^* := \mathscr Q B_0^* + \nabla \times A^* \, , \qquad B :=  \mathscr Q B_0^* + \nabla \times A \, .
\end{equation}
Then, $ (\rho,v,B^*,B) $ is a solution on $ [0,T] $ to (\ref{compressible})-(\ref{lienBB*ini}), which is associated with the initial 
data $ (\rho_0,v_0,B_0^*) $, and which is as in (\ref{sensesmooth}).
  \end{lem} 

\noindent Note that the solution to (\ref{syssimplicompot}) is no more subjected to $ \nabla \cdot v = 0 $, but we have still 
$ \nabla \cdot A^* = 0 $. The part $ \mathscr Q A $ is not involved at the level of (\ref{defineBB*autre}), though it is specified 
when solving (\ref{lienAA*inire}). In view of (\ref{lienAA*inireprem}), in general, we do not have $ \nabla \cdot A = 0 $. 

\begin{proof} By construction, we have $ B^* (0,\cdot) = B^*_0 $, and therefore $ (\rho,v,B^*)(0,\cdot) = (\rho_0,v_0,B^*_0) $
as required. On the other hand, it is clear that $ (\rho,v,B^*) $ is as indicated in (\ref{sensesmooth}). Let us consider $ B $. Since
$ \rho \, A^* \in \cH^s $, from (\ref{lienAA*inire}), we get that $ A \in \cH^{s+2} $. Then, from (\ref{defineBB*autre}), we can deduce that 
$ \nabla \times B \in \cH^s $, while by assumption $ \nabla \cdot B = \nabla \cdot B_0^* \in \cH^s $. Thus, we find that $ B \in \cH^{s+1} $
as claimed at the level of (\ref{sensesmooth}).

\smallskip

\noindent From (\ref{defineBB*autre}), we find that $ B^* -B = \nabla \times A^* - \nabla \times A $. Then, by applying the curl operator 
to (\ref{lienAA*inireprem}), we obtain (\ref{lienBB*ini}).

\smallskip

\noindent The equation on $ \rho $ is unchanged. In view of (\ref{lienAA*inireprem}) and (\ref{defineBB*autre}), the equation on $ v $ 
inside (\ref{syssimplicompot}) is just a rephrasing of the equation on $ v $ inherited from (\ref{compressible}). This also applies to the 
last equation of (\ref{syssimplicompot}) after applying the curl operator to it.
\end{proof} 

\noindent The equations inside (\ref{syssimplicompot}) bear some similarity to symmetric hyperbolic-parabolic systems which can be put 
in a normal form in the sense of Kawashima-Shizuta \cite{KS}. To see why, we have to check that the conditions enumerated in Section 3 
of \cite{KS} do apply (at least formally). To match with the notations of \cite{KS}, define $ v_I := {}^t (\rho, \mathscr P v , \mathscr P A^* ) $ 
and $ v_{II} := \mathscr Q v $. This repartition gives rise to
\begin{equation}\label{Kawashima} 
\left \{ \begin{array}{l}
\part_t v_I + A_I (v_I, v_{II} , D_x) v_I = \bar g_I (v_I,v_{II}, D_x v_{II}) \, , \\ 
\part_t v_{II} = \nu \ \nabla (\nabla \cdot  v_{II} ) + \bar g_{II} (v_I,v_{II}, D_x v_I, D_x v_{II} ) \, ,
\end{array} \right. 
\end{equation}
where $ \bar g_I = \bar g_I^1 (v_I,v_{II}, D_x v_{II}) + \bar g_I^2 (v_I,v_{II}) $ and  
$$ A_I := \text {\footnotesize 
$ \displaystyle 
\left( \begin{array}{ccc}
v \cdot \nabla & 0 & 0 \\
0 & \mathscr P \, (v \cdot \nabla)\, \mathscr P & \mathscr P \, \mathscr T^*_{A^*-A} \, \mathscr P \\
0 & \mathscr P \, \mathscr T^*_{A^*-A} \, \mathscr P & \mathscr P \, \mathscr T^*_{v-d \, (A^*-A)} \, \mathscr P
\end{array} \right) $} \, , \qquad \bar g_I^1 := \text {\footnotesize 
$ \displaystyle 
\left( \begin{array}{c}
- \rho \, \nabla \cdot v_{II} \\
- \, (v\cdot \nabla) \cdot v_{II}  \\
0
\end{array} \right) .$} $$
Recall that, given a smooth vector field $ C $, the operator $ \mathscr T_C $ is defined as in (\ref{decompoooo}) with adjoint
$ \mathscr T_C^* = - C \times ( \nabla \times \ ) $. Consider the action of $ \mathscr P \, \mathscr T^*_C \, \mathscr P $ where
the presence of $ \mathscr P $ eliminates the non symmetric one order terms (Remark \ref{yeti}). As a consequence, $ \mathscr P \, 
\mathscr T^*_C \, \mathscr P $ is (modulo zero order terms) a skew-adjoint operator with principal symbol $ i \, (C \cdot \xi) \ P(\xi) $. 
As required, $ A_I $ is skew-adjoint, the parabolic part on $ v_{II} $ is non-negative
definite, while $ \bar g_I $ depends only on $ D_x v_{II} $. This is where the role of the bulk fluid viscosity can be understood.
It is to compensate the losses of derivatives in the first equation. This idea can serve as a guide for obtaining the well-posedness. 
There are however some specific issues among which the presence of the pseudo-differential action (\ref{lienAA*inire}) 
to recover the coefficient $ A $ from $ A^* $. For the sake of completeness, we give a direct proof in the next paragraph.


\subsubsection{$ L^2 $-energy estimates} \label{paracomp3} The energy $ \cE^{\mc} $ of (\ref{compressibledecene}) is a decreasing 
quantity. Assuming that $ \rho $ remains positive (recall that the internal energy $ U $ is positive as soon as $ \rho > 0 $), 
this provides with a priori estimates on $ v $, $ \rho $ and $ \nabla \times B $. Now, from (\ref{lienAA*inireprem}) and 
(\ref{lienAA*inire}), we can deduce that
\begin{equation}\label{lienAA*inirepremtoB} 
\rho^{-1} \ \nabla \times B = ( \mathscr L^{\mc}_2)^{-1} \, \nabla \times (\nabla \times A^*) \, . 
\end{equation}
Then, by combining Proposition \ref{ellllli} and Lemma \ref{Invariant quantities}, we obtain as in the incompressible case some 
(high-frequency) $ L^2 $-bounds on $ v $ and $ A^* $. But again, this is not sufficient. We would like to have extra controls on 
$ \rho $ (other than those furnished by the integral of $ \rho \, U$) and especially stability estimates.  For these reasons, we look 
at the linearized equations which are associated with (\ref{syssimplicompot}). We think in terms of the unknowns $ (q,v, A^*) $, 
and therefore in terms of $ \dot U^{\mc}_{\mathfrak p} := ( \dot q,\dot v , \dot A^*) $. When doing this, this time, the term which is at top right of 
the second line of (\ref{syssimplicompot}) leads to
\begin{multline*}
\begin{array}{rl}  
- \, (A^*-A) \times (\nabla \times \dot A^*) + \, (\nabla \times A^* ) \times \dot A^* \! \! & - \,  (\nabla \times A^*) \times ( \dot{\mathscr L}^{\mc}_2)^{-1} \, 
(\rho \, A^*) \smallskip \\
\ & - \, (\nabla \times A^*) \times ( {\mathscr L}^{\mc}_2)^{-1} \, (\dot \rho \, A^* + \rho \, \dot A^*) \, , 
  \end{array}
\end{multline*}
where the dot on $ \mathscr L $ is needed to keep track of the dependence of $ ( {\mathscr L}^{\mc}_2)^{-1} $ on $ \rho $. Given a 
Lipschitz field $ U^{\mc}_{\mathfrak p} $, the three terms appearing in the right hand side are clearly bounded by the $ L^2 $-norm of $ \dot U^{\mc}_{\mathfrak p} $, 
and they are therefore compatible with the local $ L^2 $-stability. To simplify the presentation, they are not mentioned. Modulo source 
terms (which are ignored), we can focus on
\begin{equation}\label{syssimplicompotlinearized}  
\left \lbrace \begin{array}{l}
\part_t \dot q + v \cdot \nabla \dot q + a(q) \ \nabla \cdot \dot v = 0 \, , \medskip \\
  \displaystyle \part_t \dot v + (v\cdot \nabla ) \dot v + a(q) \ \nabla \dot q - (A^*-A) \times (\nabla \times \dot A^*) \smallskip\\
  \qquad \qquad \qquad  \ + \, \nabla \bigl( ( A^*-  A ) \cdot ( \dot A^*-  \dot A )\bigr)  = \nu \ \nabla (\nabla \cdot \dot v)  \, ,  \medskip \\
\displaystyle \part_t \dot A^* - \bigl(v - d \, (A^*-A) \bigr) \times ( \nabla \times \dot A^*) - (A^*-A) \times (\nabla \times \dot v) + 
\nabla \dot e = 0 \, .
\end{array} \right. 
\end{equation}
At the initial time $ t = 0 $, we impose
\begin{equation}\label{inidatasimplpotdotlinearized}  
\dot U^{\mc}_{\mathfrak p} (0,\cdot) = \dot U^{\mc}_{\mathfrak p 0} = (\dot q_0,\dot v_0, \dot A^*_0) \in L^2 (\RR^3;\RR) \times L^2 (\RR^3;\RR^3)^2 \, .
\end{equation}

\begin{lem} \label{L2inestimatecomp} [$L^2 $-energy estimates for the linearized incompressible  potential equations] Let $ T > 0 $.  Assume 
that  $ U^{\mc}_{\mathfrak p} = (\rho,v,A^*) $ is such that $ U^{\mc}_{\mathfrak p} \in C([0,T];H^s) $ for some $ s> 5/2 $. Then, the Cauchy problem built with 
(\ref{lienAA*inire})-(\ref{syssimplicompotlinearized}) and with initial data (\ref{inidatasimplpotdotlinearized}) has a solution on 
$ [0,T] $. Moreover, we can find a constant $ C $ depending only on the $ C([0,T];\cH^s) $-norm of $ U^{\mc}_{\mathfrak p} $ such that
\begin{equation}\label{0-energy estimatepotcomp}  
\parallel \dot U^{\mc}_{\mathfrak p} (t,\cdot) \parallel_{L^2} \leq \parallel \dot U^{\mc}_{\mathfrak p 0} \parallel_{L^2} \, e^{C \, t} \, , \qquad \forall \, t \in [0,T]  \, .
\end{equation}
\end{lem} 

\noindent Any $ \cH^s $-solution to the initial value problem (\ref{lienAA*inire})-(\ref{syssimplicompot})-(\ref{inidatasimplpotnodotcomp}) 
leads to a solution to (\ref{syssimplicompotlinearized})-(\ref{inidatasimplpotdotlinearized}) with initial data $ \dot U^{\mc}_{\mathfrak p 0} = U^{\mc}_{\mathfrak p 0} $. 
As a consequence, the proof of Lemma \ref{L2inestimatecomp} gives another access to some $ L^2 $-bound, namely
$$ \parallel U^{\mc}_{\mathfrak p} (t,\cdot) \parallel_{L^2} \leq \parallel U^{\mc}_{\mathfrak p 0} \parallel_{L^2} \, e^{C \, t} \, , \qquad \forall \, t \in [0,T]  \, . $$

\begin{proof} We multiply the first, second and third equation of (\ref{syssimplicompotlinearized}) respectively by $ \dot q $, $ \dot v $ 
and $ \dot A^* $; we integrate with respect to the variable $ x $; and then we force everywhere the emergence of $ \mathscr T_C $
by passing to the adjoint. This furnishes
 $$ \begin{array}{l} 
\displaystyle \frac{1}{2} \ \frac{d}{dt}  \bigg( \int_{\RR^3} \vert \dot U^{\mc}_{\mathfrak p} (t,\cdot)  \vert^2 \ dx  \bigg) + \nu \int_{\RR^3} \vert ( \nabla 
\cdot \dot v) (t,\cdot)  \vert^2 \ dx \leq C \int_{\RR^3} \vert \dot U^{\mc}_{\mathfrak p} (t,\cdot)  \vert^2 \ dx \smallskip \\
\qquad \qquad \qquad \qquad \qquad \quad -  \, \langle  \mathscr T_{A^*-A} \dot v , \dot A^*\rangle 
 - \langle  \mathscr T_v \dot A^*, \dot A^*\rangle +\langle  \mathscr T_{A^*-A} \dot A^* , d \, \dot A^* - \dot v \rangle \, . 
 \end{array} $$
 Knowing (\ref{ehhhhoui}), the situation is exactly as in the incompressible case, except that the divergence of $ \dot v $ is no more zero. 
 The only new term which could be problematic is issued from the first contribution in the second line. It is unavoidable in our procedure. However, it is such that 
 $$ \vert \langle (\nabla \cdot \dot v) \, (A^*-A) , \dot A^* \rangle \vert \leq \frac{\nu}{2} \, \parallel \nabla \cdot \dot v \parallel_{L^2}^2 + 
 \frac{C}{\nu} \, \parallel \dot U^{\mc}_{\mathfrak p} \parallel_{L^2}^2\, , $$
This is where the bulk (fluid) viscosity is indispensable. It serves to absorb the above loss of derivatives related to $ \nabla \cdot \dot v $. 
By Gr\"onwall's inequality, we recover (\ref{0-energy estimatepotcomp}).
\end{proof}

\noindent The comments in Paragraph \ref{paraideal2} are still appropriate. Indeed, the compressible vorticity formulation is a derived 
version of the compressible potential formulation. As such, the work of Subsection \ref{lwellposedin} can serve to confirm that 
$ \cH^1 $-estimates for the compressible potential formulation are available. This remark concludes the proof of Theorem \ref{theoprin}.


\section{Inertial wave phenomena}\label{Extended dispersion relations} 
To better grasp the role of both $ d_i $ and $ d_e $, in this section, we work with the spacetime variables of origin, those of (\ref{compressibledeb}). 
For $ 0 \leq d_e \lesssim d_i \ll 1 $ and frequencies $ \vert \xi \vert \ll d_i^{-1} $, XMHD like MHD involves principally Alv\'en and magnetosonic 
waves. The focus here is on what happens at higher frequencies, when $ \vert \xi \vert \sim d_i^{-1} $ or $ \vert \xi \vert \sim d_e^{-1} $,
while usual MHD waves may be relegated to the back burner. The emphasis is on the emergence and propagation of inertial waves. To 
simplify, we address this issue in the incompressible context, with
\begin{equation}\label{syssimplide}  
\left \lbrace \begin{array}{l}
  \displaystyle \part_t {\text v} + ({\text v} \cdot \nabla ) {\text v} + \nabla {\text p} + {\text B}^* \times (\nabla \times {\text B}) = 0 \, , \medskip \\
\displaystyle \part_t {\text B}^* + \nabla \times \bigl( {\text B}^* \times ({\text v} - d_i \,\nabla \times {\text B}) \bigr) + d_e^2 \ \nabla \times 
\bigl( (\nabla \times {\text v}) \times (\nabla \times {\text B} ) \bigr) = 0 \, ,
\end{array} \right. 
\end{equation}
together with (\ref{divfreeini}) and 
\begin{equation}\label{lienBB*2de}  
{\text B} = (\id- d_e^2 \, \Delta)^{-1} {\text B}^* \, .
\end{equation}
Our discussion is guided by the selection of different  wave configurations, aimed at revealing various facets of the analysis. 
Each time, we follow the same guidelines. First, we exhibit particular solutions to (\ref{divfreeini})-(\ref{syssimplide})-(\ref{lienBB*2de}). 
Secondly,  we derive the corresponding linearized equations (this is an opportunity to come back and complete some aspects of the 
preceding analysis). Then, we study the inertial dispersion 
relations thus generated. 

 \smallskip
 
\noindent This strategy is implemented in different situations which become somewhat more and more sophisticated. We consider 
successively: constant solutions (Subsection \ref{pannel1}), Beltrami fields (Subsection \ref{pannel2}), configurations with null points
(Subsection \ref{pannel3}), a two dimensional framework (Subsection \ref{pannel4}) and special moving solutions 
(Subsection \ref{pannel5}). 


\subsection{Constant solutions}\label{pannel1} 
Of course, constant vector fields like $ (\bar {\text v}, \bar {\text B}^*) \in \RR^3 \times \RR^3 $ give rise to solutions. The associated linearized equations
are readily identifiable
\begin{equation}\label{syssimplilinconstant}  
\left \lbrace \begin{array}{ll}
  \displaystyle \part_t \dot v + (\bar {\text v} \cdot \nabla ) \dot v + \nabla \dot p + \bar {\text B}^* \times (\nabla \times \dot B) = 0 \, ,  & \quad \nabla \cdot 
  \dot v = 0 \, , \medskip \\
\displaystyle \part_t \dot B^* + (\bar {\text v} \cdot \nabla ) \dot B^* - (\bar {\text B}^* \cdot \nabla ) ( \dot v - d_i \, \nabla \times \dot B ) = 0 \, , & \quad 
\nabla \cdot \dot B^* = 0 \, , 
\end{array} \right. 
\end{equation}
together with 
\begin{equation}\label{lienBB*2linconstant}  
\dot B = (\id- d_e^2 \, \Delta)^{-1} \dot B^* .
\end{equation}
The linear system (\ref{syssimplilinconstant}) is not symmetric (and not directly symmetrizable), confirming that the unknowns $ \dot v $ and 
$ \dot B^* $ are not suitable. Following Subsection \ref{transbis}, we can introduce the weighted vorticity $ \dot w := d_e \, \nabla \times \dot v $ 
to get
\begin{equation}\label{syssimplilinconstantvort}  
\quad \left \lbrace \begin{array}{ll}
  \displaystyle \part_t \dot w + (\bar {\text v} \cdot \nabla ) \dot w =  d_e^{-1} \, (\bar {\text B}^* \cdot d_e \, \nabla) \bigl( d_e \, \nabla \times (\id-d_e^2 \, 
  \Delta)^{-1} \dot B^* \bigr) \, , \medskip \\
\displaystyle \part_t \dot B^* + (\bar {\text v} \cdot \nabla ) \dot B^* = d_e^{-1} \, (\bar {\text B}^* \cdot d_e \, \nabla ) \bigl( \dot v - d \, d_e \, \nabla \times 
(\id- d_e^2 \, \Delta)^{-1} \dot B^* \bigr) \, .
\end{array} \right. 
\end{equation}
The derivatives of $ \dot v $ (weighted by $ d_e $) can be deduced from $ \dot w $ as indicated in Lemma \ref{elldpourw}. Observe that the 
operators which are in factor of $ d_e^{-1} $ in the right hand side are uniformly (when $ d_e \rightarrow 0 $) bounded in $ L^2 $. Thus:
\begin{itemize}
\item [-] For $ \bar {\text B}^* = O(d_e) $ or if the regime is weakly nonlinear as in (\ref{rescaledversion}), the source terms are uniformly bounded
on any finite time interval. This is the framework of the present paper.
\item [-] For $ \bar {\text B}^* = O(1) $, the $ L^2 $-norm of $ (\dot w, \dot B^*) $ may increase at a rate of $ d_e^{-1} $. This depends on the structure 
(antisymmetric or not) of the source term. As already mentioned, the corresponding effects are connected to singularity formation \cite{ChaeW,JO} or 
magnetic reconnection \cite{FP,GTAM}. These difficulties are not addressed here.
\end{itemize}

\noindent In other words, at very high frequencies $ \vert \xi \vert \gtrsim d_e^{-1} $, due to the ellipticity induced by the constitutive relation, 
all standard hyperbolic contributions (managing usually Alfv\'en and magnetosonic waves) act in the right hand side as zero order terms. 
If $ \bar {\text B}^* = O(d_e) $, they remain under control. But, for $ \bar {\text B}^* = O(1) $, they could result (when $ d_e \rightarrow 0 $) 
in a very rapid amplification of the $ L^2 $-norm.

\smallskip

\noindent The linear system (\ref{syssimplilinconstantvort}) is well-posed in $ L^2 $. However, its hyperbolic structure (the left hand side) is completely reduced,
without any influence of $ d_e $ or $ d_i $. We just find two decoupled transport equations at the velocity $ \bar v $. The constant case is a 
point of entry that does not allow to catch rich phenomena. Still, it is illustrative of the role of source terms in the inertial regime.


\subsection{Beltrami fields}\label{pannel2} Select some angular wave vector $ {\bf k} \in \RR^3 $ whose angular wavenumber 
$ k := \vert {\bf k} \vert $ is an integer ($ k \in \mathbb N $), as well as some vector $ Z_{\bf k} \in \RR^3 $ which is such that 
$ {\bf k} \cdot Z_{\bf k} = 0 $. With the help of $ {\bf k} $ and  $ Z_{\bf k} $, we can construct the oscillatory wave
$$ \cZ_{\bf k}  (x) := Z_{\bf k} \ \cos \, ({\bf k} \cdot x) + k^{-1} \ ( Z_{\bf k} \times {\bf k} ) \ \sin \, ({\bf k} \cdot x) \, . $$
This furnishes an eigenfunction of the curl operator with eigenvalue $ k $, which is called a Beltrami field. From
$$ \nabla \times \cZ_{\bf k} = k \, \cZ_{\bf k} \, , \qquad \nabla \cdot \cZ_{\bf k} = 0 \, , \qquad 2 \ (\cZ_{\bf k} \cdot \nabla) \cZ_{\bf k}  
= \nabla \vert \cZ_{\bf k}  \vert^2 \, , \qquad \Delta \cZ_{\bf k} = - k^2 \, \cZ_{\bf k} \, , $$
we can deduce that $ ({\text v},{\text B}^*) = (\cZ_{\bf k} ,\cZ_{\bf k} ) $ is a stationary solution to (\ref{divfreeini})-(\ref{syssimplide})-(\ref{lienBB*2de})
with pressure $ p = - \vert \cZ_{\bf k} \vert^2 / 2 $ and $ B = (1+ d_e^2 \, k^2)^{-1} \, \cZ_{\bf k} $. After some calculations, always with 
the weighted vorticity $ \dot w := d_e \, \nabla \times \dot v $, we find that
\begin{equation}\label{syssimplilinconstantvortb}  
\quad \left \lbrace \begin{array}{ll}
  \displaystyle \part_t \dot w + (\cZ_{\bf k} \cdot \nabla ) \dot w + d_e \, k \ (\cZ_{\bf k} \cdot \nabla ) \dot B^* = \dot {\mathscr S}^{\mi w}_{\mv 0} 
  \, (\dot w,\dot B^*) \, , \medskip \\
\displaystyle \part_t \dot B^* + (1-d_i \, k) \ (\cZ_{\bf k} \cdot \nabla ) \dot B^* + d_e \, k \ (\cZ_{\bf k} \cdot \nabla ) \dot w = \dot 
{\mathscr S}^{\mi B^*}_{\mv 0} \, (\dot w,\dot B^*) \, .
\end{array} \right. 
\end{equation}
The operators $ \dot {\mathscr S}^{\mi \star}_{\mv 0} $ are (as suggested by the notation) of order zero. They depend on $ d_e $
and $ d_i $, and they show properties similar to those identified in Subsection \ref{pannel1}. For $ d_i = 0 $, the two quantities 
$ \dot w \pm \dot B^* $ satisfy two transport equations (coupled by source terms). These inertial waves travel along the same 
characteristics, those generated by $ \cZ_{\bf k} $, but with different speeds of propagation (due to the factor $ 1\pm d_e \, k $ 
in front of $ \cZ_{\bf k} \cdot \nabla $).


\subsection{Null point configurations}\label{pannel3} The locations where the magnetic field vanishes are called
null points. Prototypes can (locally) take the form
\begin{equation}\label{prototo}   
{\text B}_\alpha^{f*} := {}^t (y,\alpha \, x,0) \, , \qquad {\text B}_\alpha^{s*} := {}^t \bigl(x,\alpha \, y,- (\alpha+1) \, z \bigr) \, , \qquad \alpha \in \RR \, . 
\end{equation}
The expression $ ({\text v},{\text B}^*) $ with $ ({\text v},{\text B}^*) = (0,{\text B}_\alpha^{f*}) $ or $ ({\text v},{\text B}^*) = 
(0,{\text B}_\alpha^{s*}) $ is a stationary solution satisfying $ {\text B}= {\text B}^* $. 

\smallskip

\noindent $ \bullet $ \underline{The} \underline{case} \underline{of} $ {\text B}_\alpha^{f*} $. First compute
$$ \nabla \times {\text B}_\alpha^{f*} = \text {\footnotesize 
$ \displaystyle 
\left( \begin{array}{c}
0 \\
0 \\
\alpha -1
\end{array} \right) $} \, , \qquad  {\text B}_\alpha^{f*} \times ( \nabla \times {\text B}_\alpha^{f*} ) = (\alpha -1) \, \text {\footnotesize 
$ \displaystyle 
\left( \begin{array}{c}
\alpha \, x \\
-y \\
0
\end{array} \right) $} = \frac{\alpha -1}{2} \ \nabla (\alpha \, x^2-y^2) \, . $$
It follows that
\begin{equation}\label{syssimplilinconstantvortbc}  
\quad \left \lbrace \begin{array}{ll}
  \displaystyle \part_t \dot w + d_e \, (\alpha-1) \ \partial_z \dot B^* = \dot {\mathscr S}^{\mi w}_{\mv 0} \, (\dot w,\dot B^*) \, , \medskip \\
\displaystyle \part_t \dot B^* - d_i \, (\alpha-1) \ \partial_z \dot B^* + d_e \, (\alpha-1) \ \partial_z \dot w = \dot {\mathscr S}^{\mi B^*}_{\mv 0} 
\, (\dot w,\dot B^*) \, .
\end{array} \right. 
\end{equation}
The inertial waves move (modulo possibly large source terms) in the vertical direction (the one of the 
stationary current density) at the speeds $ d_i \pm (d_i^2+ 4 \, d_e^2)^{1/2} \, (\alpha -1) / 2 $. In other words, two dimensional null points lend themselves to a 
transport of energy in the  direction orthogonal to the (horizontal) magnetic surfaces. This effect disappears when the perturbation remains 
in the horizontal plane or in the particular case $ \alpha = 1 $ (when the separatrix angle is $ \pi / 2 $).

\smallskip

\noindent $ \bullet $ \underline{The} \underline{case}  $ {\text B}_\alpha^{s*} $. This situation is even simpler since $ \part_t (\dot w , \dot B^* ) = 
\dot {\mathscr S}^{\mi}_{\mv 0} \, (\dot w , \dot B^* ) $.

\medskip

\noindent Large amplitude magnetic fields like in (\ref{prototo}) furnish usually templates in the perspective of reconnection models \cite{Pontin}. 
The problem is to describe what happens near the origin after perturbation. This would require (this is not done here) to measure the impact of 
the source term $ \dot {\mathscr S}^{\mi}_{\mv}  $ which is presumably of size  $ d_e^{-1} $.


\subsection{The two dimensional case}\label{pannel4} We can also seek solutions which do not depend on $ z $ and which involve 
the following form (where $ {\text B} $ and $ {\text B}^* $ are both orthogonal to $ {\text v} $)
 \begin{equation}\label{vopposite}  
 {\text v} = \text {\footnotesize 
$ \displaystyle 
\left( \begin{array}{c}
{\text v}_1(t,x,y) \\
{\text v}_2(t,x,y) \\
0
\end{array} \right) $} , \qquad  {\text B} = \text {\footnotesize 
$ \displaystyle \left( \begin{array}{c}
0 \\
0 \\
{\text b} (t,x,y) 
\end{array} \right) $} , \qquad  {\text B}^* = \text {\footnotesize 
$ \displaystyle \left( \begin{array}{c}
0 \\
0 \\
{\text b}^* (t,x,y) 
\end{array} \right) . $}
\end{equation}
Note that there exist two dimensional solutions of (\ref{syssimplide}) which are more general than (\ref{vopposite}), by including the 
flux and stream functions (see \cite{GTAM}). With (\ref{vopposite}), the equations composing (\ref{syssimplinet}) reduce to the 
following $ 2 \times 2 $ nonlinear system 
 \begin{equation}\label{syssimplinet2D}  
\left \lbrace \begin{array}{l}
\displaystyle \part_t \text{w} + v_1 \ \part_x \text{w} + v_2 \ \part_y \text{w} + d_e \ ( \part_y \text{b} \, \part_x \text{b}^* - \part_x \text{b} \, 
\part_y \text{b}^*) = 0 \, , \medskip \\
\displaystyle \part_t \text{b}^* + (v_1- d_i \, \part_y \text{b}) \ \part_x \text{b}^* + (v_2+ d_i \, \part_x \text{b}) \ \part_y \text{b}^*  + d_e \ (  \part_y 
\text{b} \, \part_x \text{w} - \part_x \text{b} \, \part_y \text{w}) = 0 \, ,
\end{array} \right. 
\end{equation}
together with
\begin{equation}\label{divfree2D}  
\part_x {\text v}_1 + \part_y {\text v}_2 = 0 \, , \qquad  \text{b} = (1- d_e^2 \, \Delta_{x,y} )^{-1} \, \text{b}^* \, .
\end{equation}
For $ d_e = 0 $, we find that $ b=b^* $, and the system (\ref{syssimplinet2D}) reduces to two transport equations
 \begin{equation}\label{syssimplinet2Dreduce}  
\left \lbrace \begin{array}{l}
\displaystyle \part_t \text{w} + v_1 \ \part_x \text{w} + v_2 \ \part_y \text{w} = 0 \, , \medskip \\
\displaystyle \part_t \text{b}^* + v_1 \ \part_x \text{b}^* + v_2 \ \part_y \text{b}^* = 0 \, .
\end{array} \right. 
\end{equation}
The Hall effects (coming from $ d_i $) just disappear (this is quite specific to this configuration). 

\smallskip

\noindent From now on, consider that $ d_e > 0 $. Then, the link between $ \text{b} $ and $ \text{b}^* $ is simplified making apparent the 
gain of two derivatives, and the role of $ \part_x \text{b} $ and 
$ \part_y \text{b} $ as coefficients. Moreover, we get simplifications since the source terms are eliminated. Fix five constants $ (\bar {\text v}_1,
\bar {\text v}_2,\alpha,\beta,\gamma) \in \RR^5 $ such that $-\bar {\text v}_1 \gamma +\bar {\text v}_2 \beta =0$. The expressions 
 \begin{equation}\label{fixdata}   
\bar {\text v} = (\bar {\text v}_1,\bar {\text v}_2) \, , \qquad (\overline{\text{w}}, \overline {\text{b}}^*) := (0, \alpha - \gamma \, x + \beta \, y) \, , \qquad \overline 
{\text{b}} = \overline {\text{b}}^* \, , 
 \end{equation}
give rise to solutions to (\ref{syssimplinet2D}) such that $ \nabla \times \bar {\text B}^* = {}^t (\beta,\gamma,0) $ is constant. The linearized 
equations of (\ref{syssimplinet2D}) along these solutions are given by
 \begin{equation}\label{syssimplinet2Dlinearizedprem}  
 \left \lbrace \begin{array}{l}
 \displaystyle \part_t \dot w + \bar {\text v}_1 \ \part_x \dot w + \bar {\text v}_2 \ \part_y \dot w  + d_e \ ( \beta \, \part_x \dot b^* + \gamma \, 
 \part_y \dot b^* ) - d_e \ ( \beta \, \part_x \dot b + \gamma \, \part_y \dot b )= 0 \, , \medskip\\
\displaystyle \part_t \dot b^* + (\bar {\text v}_1 - d_i \, \beta) \ \part_x \dot b^*  + (\bar {\text v}_2 - d_i \, \gamma) \ \part_y \dot b^* + d_e \ 
( \beta \, \part_x \dot w + \gamma \, \part_y \dot w) \\
\qquad + \, d_i \ ( \beta \, \part_x \dot b + \gamma \, \part_y \dot b ) = 0 \, ,
\end{array} \right. 
\end{equation}
where $ \dot b = (1- d_e^2 \, \Delta_{x,y} )^{-1} \, \dot b^* $. In the quasilinear symmetric presentation (\ref{syssimplinet2Dlinearizedprem}), the 
two contributions $ \beta \, \part_x \dot b^* + \gamma \, \part_y \dot b^* $ and $ \beta \, \part_x \dot w + \gamma \, \part_y \dot w $ establish a 
balance, while $ \part_x \dot b $ and $ \part_y \dot b $ are viewed as source terms (of order zero). Given some angular wave vector 
$ {\bf k} = (k_1,k_2) \in \RR^2 $ with angular wave number $ k := \vert \bf{k} \vert \in \RR_+ $ and given $ \tau \in \CC $, we can seek plane 
wave solutions of the form
 \begin{equation}\label{planewave2D}  
\dot w = \dot w_{\bf k} \ e^{i \, k_1 \, x + i \, k_2 \, y + i \, \tau \, t } , \qquad \dot b^* = \dot b^*_{\bf k} \ e^{i \, k_1 \, x + i \, k_2 \, y  + i \, \tau \, t } \, .
\end{equation}

\begin{rem} \label{compversusap} [Approximate vs complete dispersion relation] Neglecting the influence inside (\ref{syssimplinet2Dlinearizedprem})
of the zero order terms, we find the following two approximate dispersion relations
 \begin{equation}\label{approximate dispersion}  
 \tilde \tau_\pm ({\bf k}) + \bar {\text {\rm v}} \cdot {\bf k} + \frac{1}{2} \ \kappa_\pm \ ( \beta \, k_1 + \gamma \, k_2) = 0 \, , \qquad \kappa_\pm := \frac{1}{2} \ 
\Bigl(d_i \pm \sqrt{d_i^2 + 4 \, d_e^2}  \Bigr) \, ,  
\end{equation}
which are inherited from the symmetric form. As can be expected,  the functions $ \tilde \tau_\pm $ are homogeneous of degree $ 1 $ with respect to $ {\bf k} $.
\end{rem} 

\noindent Now, observe that 
$$ \beta \, \part_x \dot b^* + \gamma \, \part_y \dot b^*  - \beta \, \part_x \dot b - \gamma \, \part_y \dot b = - \, d_e^2 \ \Delta_{x,y} \, (1-d_e^2 \, 
\Delta_{x,y})^{-1} \ ( \beta \, \part_x \dot b^* + \gamma \, \part_y \dot b^* ) \, . $$
Thus, after substitution of (\ref{planewave2D}) inside (\ref{syssimplinet2Dlinearizedprem}), we get the condition $ det \, \bigl( \tau \, \id_{2 \times 2}
+ A({\bf k}) \bigr) = 0 $ where the matrix $ A({\bf k}) $ is defined by 
$$ A({\bf k}) := \bar {\text v} \cdot {\bf k} \ \id_{2 \times 2} + ( \beta \, k_1 + \gamma \, k_2) \ \left( \begin{array}{cc} 
0 & + d_e \, g({\bf k}) \\
d_e & - d_i \, g({\bf k})
\end{array} \right) \, , \qquad g({\bf k}) := \frac{d_e^2 \ k^2}{1+ d_e^2 \, k^2} \, . $$
We find two distinct \underline{real} eigenvalues giving rise to the two complete dispersion relations
 \begin{equation}\label{complete dispersion}   
 \tau_\pm ({\bf k}) +  \bar {\text v} \cdot {\bf k} + \frac{1}{2} \ ( \beta \, k_1 + \gamma \, k_2) \ \Bigl \lbrack d_i \ g ({\bf k}) \pm \sqrt{d_i^2 \, g ({\bf k})^2 
+ 4 \, d_e^2 \, g ({\bf k})} \Bigr \rbrack = 0 \, . 
\end{equation}
This means that the addition of the zero order terms does not destroy the hyperbolic properties. The $ 2 \times 2 $ system (\ref{syssimplinet2Dlinearizedprem})
is hyperbolic, with Fourier multipliers as coefficients:
\begin{itemize}
\item [-] For $ {\bf k} $ orthogonal to $ \nabla \times \bar {\text B}^* $, we just find $ \tilde \tau_\pm ({\bf k}) = \tau_\pm ({\bf k}) = - \bar {\text v} \cdot {\bf k} $.
\item [-] For $ d_e > 0 $, we must incorporate supplementary corrections on both $ \tilde \tau_\pm $ and $ \tau_\pm $ that characterize the 
propagation of inertial waves. Moreover, at the level of $ \tau_\pm $, we observe dispersive effects encoded in the (non constant)  behavior 
of $ g $. On the other hand, for $ k \gg d_e^{-1} $, we get $ g ({\bf k}) \sim 1 $. As a consequence, the asymptotic description of $ \tau_\pm ({\bf k}) $ 
gives way to $ \tilde \tau_\pm ({\bf k}) $.
\end{itemize}

\begin{rem} \label{dididi} [On the determination of the complete dispersion relations] Keep in mind that extra dispersive effects 
may be induced by the the source terms which have been skipped in this subsection.  This is here illustrated by the difference between 
$ \tilde \tau_\pm $ and $ \tau_\pm $. More generally, the eigenvalues $ \lambda_\pm $ of (\ref{vpkappa}) only provide information 
on the (maximal possible) homogeneous behavior (of order $ 1 $) inherited by the speeds of propagation (for conveniently polarized waves). 
\end{rem}


\subsection{Moving solutions}\label{pannel5} Consider that $ d_i = 1 $ and $ d_e \ll 1 $. Let $ B_0 $ be a fixed constant 
magnetic field. In \cite{AY16,ALM16}, given $ {\bf k} \in \RR^3 $,
the authors seek plane wave solutions having the following  form (where, on condition that $ {\text B}_0 = 0 $, $ {\text B} $ and $ {\text v} $ are parallel)
 \begin{equation}\label{planewave}  
\quad {\text B}= {\text B}_0 + \mu_{\bf k}^\pm \ v^\pm_{\bf k} \ e^{i \, {\bf k} \cdot x + i \, \mu_{\bf k}^\pm \, ({\text B}_0 \cdot {\bf k}) \, t } , \qquad {\text v} = 
{\text v}^\pm_{\bf k} \ e^{i \, {\bf k} \cdot x + i \, \mu_{\bf k}^\pm \, ({\text B}_0 \cdot {\bf k}) \, t } , \qquad {\text v}^\pm_{\bf k} \in \RR^3 .
\end{equation}
By adjusting the value of $ \mu_{\bf k}^\pm $ adequately, they show that such solutions do exist. Observe that $ {\text B} = {\text B}_0 + \mu_{\bf k}^\pm 
\, {\text v} $, so that $ (\nabla \times {\text B}) \times (\nabla \times {\text v} ) = 0 $. This means that the choice (\ref{planewave}) has the effect of killing 
some nonlinearities and in fact, remarkably, all nonlinearities. 

\smallskip

\noindent The choice (\ref{planewave}) is to some extent the opposite of (\ref{vopposite}). This polarization eliminates the terms which are 
emphasized at the level of (\ref{syssimplinet}) or (\ref{syssimplinet2D}), those with $ d_e $ in factor. The dynamics induced by (\ref{planewave}) 
have nothing to do with inertial waves. Rather, they are tied to some extension of Alfv\'en waves.

\smallskip

\noindent For $ {\bf k} = k \, e_z $ with $ e_z = {}^t (0,0,1) $ and where $ k = \vert {\bf k} \vert \in \RR $ stands again for the angular wavenumber, 
we get a special type of waves with associated dispersion relation
\begin{equation}\label{disprelspe}  
\omega^\pm_{\bf k} = -  \mu_{\bf k}^\pm \, ({\text B}_0 \cdot {\bf k}) = \frac{-k}{1+d_e^2 \, k^2} \ \left \lbrack - \frac{k}{2} \pm \sqrt{\frac{k^2}{4} 
+ ( 1+d_e^2 \, k^2) } \ \right \rbrack \ ({\text B}_0 \cdot e_z) \, , 
\end{equation}
which clearly exhibits dispersive properties. In Section 3.2 of \cite{ALM16}, some comments are given about (\ref{disprelspe}), which 
corresponds to a generalization of the dispersion relation for shear Alfv\'en waves in ideal MHD. Since the $ \omega^\pm_{\bf k} $ remain
bounded, the role of electron inertia in this case is to impose a lower and upper bound on the time frequencies attainable. In contrast,
in the Hall framework, we get $ \omega^-_{\bf k} = k^2 \, (B_0 \cdot e_z) $ which rapidly diverges as the spatial wavenumber $ k $ tends 
to infinity. This means that the electron inertia has the effect on $ \mu_{\bf k}^- $ to cure singular behaviors at high wave numbers in Hall MHD. 
  

\section{Appendix} \label{Appendix}
 In Subsection \ref{listidentities}, we list some useful identities implying $ \nabla \times $.
In Subsection \ref{div-curlsystem}, we recall elliptic $ L^2 $-estimates concerning the div-curl system. These estimates have been exploited
to control the derivatives of $ v $ in terms of $ \nabla \times v $ and $ \nabla \cdot v $. 


\subsection{Identities involving the curl operator} \label{listidentities} Retain that
\begin{equation}\label{curlcurl}  
\nabla \times ( \nabla \times v) = \nabla ( \nabla \cdot v) - \Delta v \, , 
\end{equation}
We need to know that
\begin{equation}\label{followingidentity}   
\nabla \times (F \times G) = \bigl( (\nabla \cdot G) + G \cdot \nabla \bigl) F - \bigl( (\nabla \cdot F) + F \cdot \nabla \bigl) G \, . 
\end{equation}
Recall also that
\begin{equation}
  \label{VecIdent-1}
\begin{array}{rl}
 \nabla \times (F\cdot \nabla G) = \! \! \! & (F\cdot\nabla)\nabla \times G-((\nabla \times G)\cdot \nabla)F \\
 \ & \displaystyle + (\nabla \times G) (\nabla \cdot F) + \sum_{i=1}^3\nabla F_i \times \nabla G_i \,.
\end{array}
\end{equation}


\subsection{Elliptic estimates for the div-curl system} \label{div-curlsystem} Consider in $ \RR^3 $ the system
\begin{equation}\label{elldivcurl}  
\left \lbrace \begin{array}{l}
  \displaystyle \nabla \times v = w \, , \\
\displaystyle \nabla \cdot v = g \, ,
\end{array} \right. 
\end{equation}
where $ w $ and $ g $ are given data in $ L^2 $, whereas $ v $ is the unknown. From (\ref{curlcurl} ), it is easy to infer that
$$ \sum_{i,j}^3 \int_{\RR^3} \vert \part_i v_j \vert^2 \ dx =  \int_{\RR^3} \bigl( \vert \nabla \cdot v \vert^2 +  \vert \nabla \times v \vert^2 
\bigr) \ dx = \int_{\RR^3} \bigl( \vert w \vert^2 +  g^2 \bigr) \ dx \, . $$
The derivatives of $ v $ in $ L^2 $ are therefore controlled by the $ L^2 $-norms of $ w $ and $ g $. Using the Poincar\'e--Sobolev inequality 
(i.e. $\| v\|_{L^2(\RR^3)} \leq C \| \nabla v\|_{L^2(\RR^3)}$) we also  obtain the control of the $L^2$-norm of $v$; hence its $\cH^1$-norm.
Below, we formalize this well-known fact \cite{Wahl}. For the sake of completeness, we also give a more explicit proof of it.


\subsubsection{Link between the vorticity and the derivatives of a divergence free velocity} \label{linkwv} We start by manipulating 
solenoidal vector fields, belonging to $ \cD^s $. The link between $ w $ and $ v $ is then achieved through the Biot-Savart law. 
\begin{equation}\label{Biot-Savart}  
v = \nabla \times (-\Delta)^{-1} \, w \, . 
\end{equation}

\begin{lem} \label{elldpourw} [Continuity properties when passing from $ \nabla \times v $ to $ \part_i v $] Given $ w \in \cD^s $, there exists 
a unique solenoidal vector field $ v $ such that $ w = \nabla \times v $ in the distributional sense. Moreover, for all $ i \in \{1,2,3 \} $, the linear 
operator $ \cM_i^\mi : \cD^s \rightarrow \cD^s $ which sends $ w $ to $ \part_i v $ (with $ v $ as above) may be defined as a bounded matrix Fourier 
multiplier. It is therefore continuous for all $ s \in \RR $.
\end{lem} 

\begin{proof} Fix any $ w \in \cD^s $. By Poincar\'e lemma, we can find some $ v $ such that $ w = \nabla \times v $ in the distributional sense. 
If we impose moreover $ \nabla \cdot v = 0 $, on the Fourier side, we have to deal with the explicit relation $ \hat v = \cF v = i \, \vert \xi \vert^{-2} \, 
\cF w \times \xi $, which furnishes 
$$ \widehat{\part_i v} = M_i^\mi (\xi) \, \cF w \, , \qquad M_i^\mi (\xi):= - \frac{\xi_i}{\vert \xi \vert^2} \, \left( \begin{array}{ccc}
0 & \xi_3 & - \xi_2 \\
- \xi_3 & 0 & \xi_1 \\
\xi_2 & - \xi_1 & 0 
\end{array} \right) . $$
It is clear that the matrix-valued function $ M_i^\mi $ is bounded on $ \RR^3 \setminus \{0\} $.
\end{proof}


\subsubsection{Link between $ (w,g) $ and the derivatives of $ v $} \label{linkwvc} The compressible version of Lemma \ref{elldpourw} is the 
following.

\begin{lem} \label{elldpourdivw} [Continuity properties when passing from the couple $ (\nabla \times v, \nabla \cdot v) $ to $ \part_i v $] Let 
$ g \in \cH^s (\RR^3;\RR) $ and $ w \in \cD^s $. There exists a unique $ v $ such that $ (\nabla \times v,\nabla \cdot v) = (g,w) $ in the distributional 
sense. Moreover, for all $ i \in \{1,2,3 \} $, the linear operator $ \cM_i^{\mc} : \cH^s \times \cH^s \rightarrow \cH^s $ which sends $ (g, w) $ to $ \part_i v $ (with 
$ v $ as above) may be defined as a bounded matrix Fourier multiplier. It is therefore continuous.
\end{lem} 

\begin{proof} By construction, we have
$$ \cM_i^{\mc} \, (g,w) = \cF^{-1} \bigg( -\frac{\xi_i}{\vert \xi \vert^2} \ ( \hat g \ \xi  - \xi \times \hat w ) \bigg) \, , $$
which is sufficient to conclude.
\end{proof}

\end{document}